\documentclass{article}

\usepackage[utf8]{inputenc}
\usepackage[english]{babel}
\usepackage[centering]{geometry}

\usepackage{amsmath}
\usepackage{amsthm}
\usepackage{amssymb}
\usepackage{amsfonts}

\usepackage{bm} %% for identitiy \bbmath 1

\usepackage{color}
\usepackage{graphicx}
\usepackage{subcaption}
\usepackage{enumitem} 
\usepackage{comment} 

\usepackage{fancyhdr}
\usepackage{mathtools}

\usepackage{tikz}
\usetikzlibrary{arrows}
\usepackage{fp}  		%%% for arithmetics
\usetikzlibrary{calc, shapes, positioning, decorations}
\usepackage{pgfplots}
\usetikzlibrary{math} % tikz library to define variables
\usetikzlibrary{intersections}  %% 
\pgfplotsset{compat=newest}

\usetikzlibrary{external}
\DeclareMathAlphabet{\bbold}{U}{bbold}{m}{n}  % for \vmathbb{1} for identity matrix
\newcommand{\id}{\ensuremath{\bbold{1} }}     % for \vmathbb{1} for identity matrix

\usepackage[hypertexnames=false, breaklinks,
   colorlinks=true, 
   citecolor=red!70!black,
   linkcolor=blue,
   anchorcolor=cyan,
   urlcolor=green!50!black,
   pagebackref=false,
   ]{hyperref}

\usepackage[breakable]{tcolorbox}
\tcbset{shield externalize} 
\tcbuselibrary{theorems}
\newtcolorbox{bluebox}[1][]%
{left=0mm, right=0mm, bottom=0mm, top=0mm, sharp corners, boxrule=.8pt, before skip=\topsep, after skip=\topsep, colback=cyan!5, colframe=cyan, coltitle=black, fonttitle=\bfseries, title=#1, breakable}
\tcbset{tcbox raise base, sharp corners, left=0mm,right=0mm,top=0mm,bottom=0mm,nobeforeafter, boxrule =.5pt} 
\tcbset{tcbox raise base, sharp corners, left=0mm,right=0mm,top=0mm,bottom=0mm,nobeforeafter, boxrule =.5pt}

\usepackage[color=yellow!30, linecolor=orange, colorinlistoftodos]{todonotes}

\newcommand{\abs}[1]{\left| #1\right|}

\theoremstyle{plain}
\newtheorem{theorem}{Theorem}[section]
\newtheorem*{theorem*}{Theorem}
\newtheorem{proposition}[theorem]{Proposition}

\newtheorem*{lemma*}{Lemma}
\newtheorem{corollary}[theorem]{Corollary}

\theoremstyle{definition}

\newtheorem{remark}[theorem]{Remark}
\newtheorem*{remark*}{Remark}

\newtheorem{example}[theorem]{Example}

\newcommand{\R}{\mathbb{R}}
\newcommand{\C}{\mathbb{C}}

\newcommand{\N}{\mathbb{N}}
\newcommand{\mD}{\mathcal{D}}
\newcommand{\mH}{\mathcal{H}}
\newcommand{\mL}{\mathcal{L}}
\newcommand{\mT}{\mathcal{T}}

\newcommand{\rd}{\mathrm d}

\newcommand{\mfs}{\mathfrak s}

\newcommand{\nuext}{\nu_{\mathrm{ext} }}
\newcommand{\nuint}{\nu_{\mathrm{int} }}
\newcommand{\Msup}{\widehat m}
\newcommand{\Mlim}{M}
\providecommand{\norm}[1]{\left\lVert#1\right\rVert}
\newcommand{\scalar}[2]{\langle #1, #2\rangle}

\newcommand{\hp}[2]{\left\langle #1, #2\right\rangle}
\providecommand{\ceil}[1]{\left\lceil#1\right\rceil}
\providecommand{\floor}[1]{\left\lfloor#1\right\rfloor}

\DeclareMathOperator{\re}{Re}%Parte real\textbf{}\textbf{}

\DeclareMathOperator{\sgn}{sgn}
\DeclareMathOperator{\linspan}{span}

\newcommand{\define}[1]{\emph{#1}}
\newcommand{\essspec}{\sigma_{\mathrm{ess}}}    %% essential spectrum
\newcommand{\discspec}{\sigma_{\mathrm{disc}}}  %% discrete spectrum
\newcommand{\pointspec}{\sigma_{\mathrm{p}}}  %% point spectrum
\newcommand{\acspec}{\sigma_{\mathrm{ac}}}
\begin{document}

\title{On the existence of eigenvalues of a one-dimensional Dirac operator}
\author{Daniel Sánchez-Mendoza\thanks{ETSI Navales, Universidad Politécnica de Madrid, España. Email: daniel.sanchezmen@upm.es},
Monika Winklmeier\thanks{Departamento de Matem\'aticas, Universidad de los Andes, Bogot\'a.}
}
\maketitle

\begin{abstract}
    The aim of this paper is to study the existence of eigenvalues in the gap of the essential spectrum of the one-dimensional Dirac operator in the presence of a bounded potential.
    We employ a generalized variational principle to prove existence of such eigenvalues, estimate how many eigenvalues there are, and give upper and lower bounds for them.
\end{abstract}

\section{Introduction} %%%{{{

In this paper we study the one-dimensional Dirac operator given by the differential expression
\begin{equation}
\label{eq:DiracDefIntro}
    \tau  = 
   \begin{pmatrix}
      0 & \frac{\rd}{\rd x} \\
      -\frac{\rd}{\rd x}  & 0
   \end{pmatrix}
   + V(x)
   =
   \begin{pmatrix}
      0 & \frac{\rd}{\rd x} \\
      -\frac{\rd}{\rd x}  & 0
   \end{pmatrix}
   +
   \begin{pmatrix}
      M_1(x) & W(x) \\
      W(x) & -M_2(x)
   \end{pmatrix}
\end{equation}
for $x\in\R$ or $x\in \R_+$. Such operators are used in quantum mechanics to describe fermionic particles.
Our main assumptions are that $M_1,\, M_2$ and $W$ are real valued bounded measurable functions whose limits for $|x|\to \infty$ exist, 
and that $\inf\{M_1\} > \sup\{-M_2\}$.
It is well known that the minimal operator associated with $\tau$ on $\R$ is essentially selfadjoint.
If we consider $\tau$ on the halfline $\R_+$, then it is symmetric and has defect indices $(1,1)$, so we need a boundary condition at $0$ in order to obtain a selfadjoint operator, see Proposition~\ref{prop:defT}. Let us denote by $T$ the selfadjoint realisation of $\tau$ on $\R$. Then the eigenvalues of the matrix $\lim_{x\to\pm\infty} V(x)$ determine the essential spectrum of the Dirac operator which turns out to be of the form $\essspec(T) = (-\infty, \lambda_{e-}] \cup [\lambda_{e+}, \infty)$, while the gap $(-m_2, m_1)$ between the ranges of $M_1$ and $-M_2$ belongs to its resolvent set $\rho(T)$, see Proposition~\ref{prop:essspecT} and Proposition~\ref{prop:EssspecResT}.
Therefore, discrete eigenvalues of $T$ may appear only in $(\lambda_{e-}, -m_2]\cup [m_1, \lambda_{e+})$.
Also for selfadjoint realisations of $\tau$ on $\R_+$ with boundary condition parameter $\alpha$, denoted $T_{\alpha+}$, the essential spectrum is determined by the limit of $V$ at infinity.
Moreover, $(-m_2, m_1)\subseteq\rho(T_{\alpha+})$
provided that $\alpha\in\{0,\pi/2\}$. For a different choice of $\alpha$ this is not necessarily true, 
see for instance Remark~\ref{rem:EigenvalueAlpha}.
\smallskip

Operators given by \eqref{eq:DiracDefIntro} appear, for example, as the radial part of the three-dimensional Dirac operator with a radial symmetric potential after a separation ansatz.
More precisely, the so called partial wave operators are given by, see~\cite[Thm. 4.14]{thaller},
\begin{equation*}
    h_{m_j, \kappa_j}
    =
    \begin{pmatrix}
        mc^2 + \Phi_{sc}(x) + \Phi_{el}(x) & c\{- \frac{\rd}{\rd x} + \frac{\kappa_j}{x} \} + \Phi_{am}(x) \\[1ex]
        c\{ \frac{\rd}{\rd x} + \frac{\kappa_j}{x} \} + \Phi_{am}(x)  & -mc^2 - \Phi_{sc}(x) + \Phi_{el}(x) 
    \end{pmatrix}
\end{equation*}
for $x\in \R_+$.
The parameters $j\in \{\frac{1}{2},\, \frac{3}{2},\, \dots\}$, $m_j \in \{-j,\, -j+1,\, \dots, j\}$, $\kappa_j = \pm (j+\frac{1}{2})$ stem from the separation ansatz and describe the angular moment.
The constants $m$ and $c$ are the mass of the particle and the velocity of light.
The functions $\Phi_{sc}$, $\Phi_{el}$,  $\Phi_{am}$ are interpreted as scalar potential, electric potential and the anomalous magnetic moment.
If we set $c=1$ and apply the unitary transformation 
$U = \begin{psmallmatrix} 1 & 0 \\ 0 & -1 \end{psmallmatrix}$
gives $Uh_{m_j, \kappa_j} U^{*} = \tau$
with
$M_1(x) = m + \Phi_{sc}(x) + \Phi_{el}(x)$,
$M_2(x) = m + \Phi_{sc}(x) - \Phi_{el}(x)$,
$W(x) = -\frac{\kappa_j}{x} - \Phi_{am}(x)$.
Since in this work we always assume that $M_j$ and $W$ are bounded, we can deal with this operator only in the unphysical situation when $\kappa = 0$. We will remove the boundness hypothesis of $M_j$ and $W$ in a future work so we can handle the physically relevant cases.
\smallskip

Operators of the form \eqref{eq:DiracDefIntro} arise also in the description of Dirac fermions in graphene, see for instance \cite{JakubskyKrejcirik2014} and references therein.
Jakubsk\'{y} and Krej\v{c}i\v{r}\'{\i}k studied their spectrum under the assumption that $M_1(x) = M_2(x) = M\geq0$. 
In this case, $M$ and $\pm \frac{\rd}{\rd x} + W$ commute and the square of the Dirac operator is a diagonal operator.
If $W$ is differentiable, then the entries on the diagonal are the Schr\"odinger operators
$H_\pm = M^2 +(\pm\frac{\rd}{\rd x} + W)(\mp\frac{\rd}{\rd x} + W) = M^2 - \frac{\rd^2}{\rd x^2} + W^2 \pm W'$.
With the help of the min-max principle for semibounded operators, they could show existence of eigenvalues of $H_\pm$ and then conclude the existence of eigenvalues of the Dirac operator in the gap of its essential spectrum.
See also Remark~\ref{rem:ADequalMM}.
\smallskip

Recently, there has been interest in one-dimensional Dirac operators which are not derived from the Dirac operator in higher dimensions, see \cite{DasguptaKhuranaTahvildarZadeh2023}.  
The authors consider a one-dimensional Dirac fermion on the real line with a screened Coulomb potential and study its eigenvalues in the gap of its essential spectrum.
In their case, $W(x)=0$, $M_1(x) = m - e\Phi(x)$, $M_2(x) = m + e\Phi(x)$ where
$\Phi(x) = \frac{Q}{2}\mu e^{-|x|/\mu}$.
They transform the Dirac eigenvalue equation expression into an autonomous dynamical system and show that for suitable choices of the physical constants, eigenvalues appear in the gap of the essential spectrum.
\smallskip

In the work \cite{KramerKrejcirik2024} the authors study the Dirac operator $T_{\alpha+}$ on the halfline $\R_+$ with a possibly unbounded and complex potential $V(x)=\begin{psmallmatrix}M&0\\0&-M\end{psmallmatrix}+Q(x)$ for some $M>0$, and ask when it does not have eigenvalues. 
They find sufficient conditions on $Q$ and $\alpha$ which guarantee that $T_{\alpha,+}$ does not have eigenvalues.
While in our work we need the boundary value parameter 
$\alpha$ to be equal to $0$ or $\frac{\pi}{2}$ so that the variational principle works, 
the possible values of $\alpha$ that are allowed in \cite{KramerKrejcirik2024} are $\alpha \in (0, \pi/2)$. 
The reason for this is that their method relies on an 
explicit formula for the resolvent of the free Dirac operator $(T_{\alpha,+})|_{Q=0}$ which does not hold otherwise.
In fact, the free Dirac operator has a discrete eigenvalue for $\alpha\in(\frac{\pi}{2}, \pi)$.
See Proposition~\ref{prop:defT} and Remark~\ref{rem:EigenvalueAlpha}.
\smallskip

In order to obtain information on the discrete spectrum of a linear semibounded operator, very often a min-max or max-min principle is used. We now state these classical variational principles, for the proof see \cite[Thm. 12.1]{schmuedgen} or \cite[Thm.~XIII.2]{ReedSimon4}. 
Let $A$ be a lower semibounded selfadjoint linear operator on a Hilbert space $H$ with associated sesquilinear form $\mathfrak a$ and set $\lambda_e \coloneqq\inf \essspec(A)$. Then $\sigma(A)\cap(-\infty,\lambda_e)$ consists of at most countably many eigenvalues of finite multiplicities which may accumulate only at $\lambda_{e}$. Let $(\lambda_n)_{n=1}^\infty$ be the increasing sequence of these eigenvalues counted with multiplicity, extended by $\lambda_{e}$ if there are only finitely many. Then
\begin{equation}
   \label{eq:classicMinMax}
   \lambda_n = 
   \min_{ \substack{ L\subseteq\mD(\mathfrak a) \\ \dim L = n} }
   \max_{ \substack{ x\in L \\ x\neq 0} } p(x)
   =
   \max_{ \substack{ L\subseteq H \\ \dim L = n-1 }}
   \min_{ \substack{ x\in \mD(\mathfrak a)\cap L^\perp \\ x\neq 0} }
   p(x),\qquad n\in \N,
\end{equation}
where $p(x)$ is the Rayleigh quotient
\begin{equation*}
    p(x) \coloneqq \frac{\mathfrak a[x,x]}{\|x\|^2},
    \qquad x\in \mD(\mathfrak a)\setminus\{0\}.
\end{equation*}
\begin{remark}
   For every $\lambda\in\R$ we can define the quadratic form 
   \begin{equation}
   \label{eq:classicalzeros}
       \mathfrak a(\lambda)[x] \coloneqq \mathfrak a[x,x] - \lambda \|x\|^2,\qquad x\in \mD(\mathfrak a)\setminus\{0\}.
   \end{equation}
   Clearly, for fixed $x\in \mD(\mathfrak a)\setminus\{0\}$, the function $\lambda\mapsto a(\lambda)[x]$
   is decreasing and has a unique zero equal to $p(x)$.
   Hence the classical variational principles in \eqref{eq:classicMinMax} can be written as min-max or max-min over the zeros of the function in \eqref{eq:classicalzeros}.
\end{remark}
However, the classical variational principles do not work for the discrete spectrum in a gap of the essential spectrum of a selfadjoint operator. This is because the min-max or max-min expression in \eqref{eq:classicMinMax} will always be less than or equal to the infimum of the essentials spectrum. In particular, for our Dirac operator these expression will always give $-\infty$.
\smallskip

We employ a generalized variational principle in order to obtain information on the eigenvalues of the Dirac operator that are within the essential spectral gap, see \eqref{eq:def:mun} and Theorem~\ref{thm:varprinciple:abstract}. For the generalized variational principle to be applicable, the operator must be a block operator matrix whose domain is the direct sum of subspaces. This condition is fulfilled by the operator $T$ on $\R$, as well as $T_{\alpha+}$ on $\R_+$ if $\alpha\in\{0,\pi/2\}$.
Formally the generalized variational principle looks like the classical one but instead of the zeros of \eqref{eq:classicalzeros} the zeros of a more complicated functional appear in \eqref{eq:classicMinMax}.
Such principles have been investigated in \cite{BEL2000}, \cite{EL04}.
Applications to Dirac operators can be found e.g. in \cite{KLT04}, \cite{winklmeier}.
See also the book \cite{tretter}.
\smallskip

In Section~\ref{sec:SetUp} we will define the one-dimensional Dirac operator on $\R$ and on the halflines $\R_\pm$, set our assumptions on the potential, and fix notation.
In Section~\ref{sec:varprinciple} we state the variational principle in an abstract setting and obtain a first sufficient condition for the existence of at least one eigenvalue in the essential spectral gap (Corollary~\ref{cor:abstractEVexistencis}).
Moreover, the special case of constant $M_1 = M_2$ is discussed in Proposition~\ref{prop:ADequalMM}.
In Sections~\ref{sec:DiracRealLine} and \ref{sec:DiracHalfLine} we derive from the abstract variational principle estimates on the number and location of eigenvalues in the essential spectral gap of the Dirac operator on $\R$ and on $\R_+$.
The main theorems are Theorem~\ref{thm:teoremafavoritoglobal} where the eigenvalues of the Dirac operator are compared to those of an associated Sturm-Liouville operator,
and Theorem~\ref{thm:teoremafavorito} which can be considered a localized version of the previous one.
In particular, it shows that deep enough valleys in $W$ produce eigenvalues of the Dirac operators.
We test our results on a toy example that can be solved analytically (Example~\ref{ex:Analytic}) and on the one-dimensional hydrogenic Dirac operator studied in \cite{DasguptaKhuranaTahvildarZadeh2023}. 
We find that our numerical results agree very well with theirs.
\smallskip

{\bfseries Notation.}
   Let $A$ be a closed linear operator defined on a Hilbert space $H$.
   We denote its domain, resolvent and spectrum by $\mD(A)$, $\rho(A)$ and $\sigma(A)$. We use $\pointspec(A)$ and $\discspec(A)$ for its point and discrete spectrum; and for its essential spectrum we use the definition from \cite[Sec. XVII.2]{GohGolMarKaa90}
   \begin{align*}
   \essspec(A)\coloneqq\{\lambda\in\C : A-\lambda \text{ is not Fredholm}\}.
   \end{align*}
   If $A$ is selfadjoint, the spectral subspace associated to an interval $I$ is denoted by $\mathcal L_{I}A$.
   For a pencil of closed linear operators $(S(\lambda))_{\lambda\in\Omega}$, where $\Omega$ an open subset of $\C$, we define its resolvent set and its spectrum by
   \begin{align*}
    \rho(S) &= \{\lambda\in\Omega : 0\in \rho(S(\lambda)) \},
    \\
    \sigma(S) &= \{\lambda\in\Omega : 0\in \sigma(S(\lambda)) \}.
   \end{align*}
   The sets $\pointspec(S)$, $\discspec(S)$ and $\essspec(S)$ are defined similarly.
%%%}}}

\section{The one-dimensional Dirac operator}
\label{sec:SetUp}%%%{{{

In this section we define the one-dimensional Dirac operator given by
\begin{equation*}
   \tau  = 
   \begin{pmatrix}
      0 & \frac{\rd}{\rd x} \\
      -\frac{\rd}{\rd x}  & 0
   \end{pmatrix}
   + V(x)
   =
   \begin{pmatrix}
      0 & \frac{\rd}{\rd x} \\
      -\frac{\rd}{\rd x}  & 0
   \end{pmatrix}
   +
   \begin{pmatrix}
      M_1(x) & W(x) \\
      W(x) & -M_2(x)
   \end{pmatrix}
    \end{equation*}
on $\R$ and on $\R_\pm$, as well as some relevant quantities for its spectrum.
% Such operators for $x\in\R_+$ arise for example from the usual Dirac operator in $\R^3$ with a radial symmetric potential after a separation ansatz, see e.g. \cite[Sec. 4.6.6]{thaller}.
We will always assume the following conditions for the Dirac operator defined on $\R$, with the obvious modifications for $\R_\pm$:

\begin{enumerate}[label={\upshape (A\arabic*)}]

   \item\label{item:A1}
   $M_1, M_2, W$ are bounded measurable real valued functions.

   \item\label{item:A2}
   $-m_2 < m_1$ where
   $m_i \coloneqq \displaystyle \inf_{x\in\R} M_i(x)$\ for $i=1,2$.

   \item\label{item:A3}
   The following limits exist:
   \begin{equation*}
      W_{\pm} \coloneqq \lim_{x\to\pm\infty} W(x),
      \quad
      \Mlim_{j\pm} \coloneqq \lim_{x\to\pm\infty} M_j(x).
   \end{equation*}
\end{enumerate}

We set
\begin{equation*}
   V_\pm \coloneqq \lim_{x\to\pm\infty} V(x)=
   \begin{pmatrix}
      \Mlim_{1\pm} & W_\pm \\
      W_\pm & -\Mlim_{2\pm}
   \end{pmatrix}.
\end{equation*}
Clearly, the eigenvalues of $V_+$ are
\begin{equation*}
   \lambda_+^{\pm} = \frac{\Mlim_{1+} - \Mlim_{2+}}{2} \pm \sqrt{ \left( \frac{\Mlim_{1+} + \Mlim_{2+}}{2} \right)^2 + W_+^2}\ ,
\end{equation*}
and those of $V_-$ are 
\begin{equation*}
   \lambda_-^{\pm} = \frac{\Mlim_{1-} - \Mlim_{2-}}{2} \pm \sqrt{ \left( \frac{\Mlim_{1-} + \Mlim_{2-}}{2} \right)^2 + W_-^2}\ .
\end{equation*}
From these eigenvalues we define
\begin{equation}\label{eq:def:lambda_e}
    \lambda_{e-} \coloneqq \max\{ \lambda_+^-,\  \lambda_-^- \},
    \qquad
    \lambda_{e+} \coloneqq \min\{ \lambda_+^+,\  \lambda_-^+ \},
\end{equation}
and note that \ref{item:A2} and \eqref{eq:def:lambda_e} imply that 
\begin{equation}
\label{eq:lambdainequalities}
   \lambda_\pm^-\leq\lambda_{e-} \leq -m_2 <m_1\leq \lambda_{e+}\leq \lambda_\pm^+.
\end{equation}
In the special case $M_1 = M_2$, the above simplifies to
\begin{equation*}
    \lambda_+^\pm = \pm\sqrt{ \Mlim_{1+}^2 + W_+^2 },
    \qquad 
    \lambda_-^\pm = \pm\sqrt{ \Mlim_{1-}^2 + W_-^2 }
\end{equation*}
and 
$\lambda_{e+} = -\lambda_{e-} = \min\left\{ \sqrt{ \Mlim_{1-}^2 + W_-^2 }\, ,\, \sqrt{ \Mlim_{1+}^2 + W_+^2 } \right\}$.
\smallskip

Now we define $\tau$ as an operator on the Hilbert spaces $L_2(\R)^2$ and $L_2(\R_\pm)^2$
\begin{proposition}
   \label{prop:defT}
   Assume that 
   \ref{item:A1}, \ref{item:A2}, \ref{item:A3} hold.
   \begin{enumerate}[label={\upshape (\roman*)}]
      \item
      The operator 
      \begin{equation*}
	 T = 
      \begin{pmatrix}
        0 & \frac{\rd}{\rd x} \\
        -\frac{\rd}{\rd x} & 0
	 \end{pmatrix}
	 + V
    =
	 \begin{pmatrix}
	    M_1 & \frac{\rd}{\rd x} + W \\
	    -\frac{\rd}{\rd x} + W & -M_2
	 \end{pmatrix}
      \end{equation*}
      with domain
      \begin{equation*}
	 \mD(T) = \left\{ 
	 \Psi \in L_2(\R)^2 :
	 \Psi\text{ is absolutely continuous, } \tau\Psi \in L_2(\R)^2
	 \right\}
	 = H^1(\R)^2
      \end{equation*}
      is selfadjoint and all its eigenvalues are simple.

      \item 
      For every $\alpha\in [0,\pi)$ the operator 
      \begin{equation*}
	 T_{\alpha\pm} = 
	 \begin{pmatrix}
        0 & \frac{\rd}{\rd x} \\
        -\frac{\rd}{\rd x} & 0
	 \end{pmatrix}
	 + V
    =
	 \begin{pmatrix}
	    M_1 & \frac{\rd}{\rd x} + W \\
	    -\frac{\rd}{\rd x} + W & -M_2
	 \end{pmatrix}
      \end{equation*}
      with domain
      \begin{align*}
	 \mD(T_{\alpha\pm}) 
	 & = \left\{ 
	 \Psi \in L_2(\R_{\pm})^2 :
      \begin{aligned}
	    & \Psi\text{ is absolutely continuous, } \tau\Psi \in L_2(\R_{\pm})^2,\\
	    & \Psi(0) = c
	    \begin{pmatrix}
	    \cos\alpha \\ \sin\alpha
	    \end{pmatrix}\text{ for some } c\in\C
      \end{aligned}
	 \right\}
	 \\[2ex]
	 & = \left\{ 
	 \Psi \in H^1(\R_{\pm})^2 :
	 \Psi(0) = c
	 \begin{pmatrix}
	    \cos\alpha \\ \sin\alpha
	 \end{pmatrix}\text{ for some } c\in\C
	 \right\}
      \end{align*}
      is selfadjoint and all its eigenvalues are simple.
      All selfadjoint realizations of $\tau$ on $\R_\pm$ are of this form.

   \end{enumerate}

\end{proposition}
\begin{proof}
By \cite[Thm. 6.8]{weidmann1987} $\tau$ is in the limit point case at $\pm\infty$, which implies that any selfadjoint realization on $\R$ or $\R_\pm$ has simple eigenvalues.
\smallskip

Note that the operator $T$ defined above is the maximal symmetric operator associated to $\tau$ on $\R$. By \cite[Thm. 5.7]{weidmann1987}
$\tau$ admits only one selfadjoint realization, which is $T$ by its maximality.
\smallskip

If we consider $\tau$ on $\R_+$, then due to $V$ being bounded, $\tau$ is in the limit circle case at $0$. Therefore all selfadjoint realizations are given by fixing an appropriate boundary condition at $0$. In fact, they are given by the family $(T_{\alpha+})_{\alpha\in[0,\pi)}$ as defined above, see \cite[Thm. 5.8]{weidmann1987}
or \cite[Beispiel~15.26]{weidmannII}.
The case of $\R_-$ can be treated similarly.
\end{proof}

Next we recall that the limit behaviour of $V$ for $x\to\pm\infty$ determines the essential spectrum of $T$ and $T_{\alpha\pm}$.
\begin{proposition}\label{prop:essspecT}
   Let $T$ and $T_{\alpha\pm}$ be as in Proposition~\ref{prop:defT}.
   \begin{enumerate}[label={\upshape (\roman*)}]

      \item 
      $\essspec(T_{\alpha\pm}) = (-\infty,\, \lambda_\pm^-] \cup [\lambda_\pm^+,\, \infty)$.

      \item 
      $\essspec(T) = (-\infty,\, \lambda_{e-}] \cup [\lambda_{e+},\, \infty)$.
   \end{enumerate}
\end{proposition}
\begin{proof}
    
   From \cite[Thm. 16.5 and 16.6]{weidmann1987} we have that 
   $\essspec(T_{\alpha\pm}) = (-\infty, \lambda_\pm^{-}] \cup [\lambda_\pm^{+}, \infty)$.
   Note that $T$ and $T_{0-}\oplus T_{0+}$ are two-dimensional selfadjoint extensions of the symmetric operator $T|_{\mD_0}$ with $\mD_0 = \{ \Psi\in \mD(T) : \Psi(0) = 0 \}$.
   Therefore 
   \begin{align*}
      \essspec(T) = 
      \essspec(T_{0-}\oplus T_{0+}) 
      = \essspec(T_{0-}) \cup \essspec(T_{0+})
      = (-\infty,\, \lambda_{e-}] \cup [\lambda_{e+},\, \infty).
      & \qedhere
   \end{align*}
\end{proof}

\begin{remark}
   \label{rem:AbsCont}
   If $V$ can be written as $V=V_1+V_2$ with $V_1\in L_1(\R_+)$ and
   $V_2$ of bounded variation and 
   converging to a diagonal matrix as $x\to\infty$, then $(-\infty,\, \lambda_+^-) \cup (\lambda_+^+,\, \infty)\subseteq\acspec(T_{\alpha+})$ by \cite[Thm.~16.7]{weidmann1987}.
   This is the case if for instance $M_1,M_2$ are of bounded variation and $W(x)\to 0$ as $x\to\infty$ fast enough.

   Similar assertions hold for $T_{\alpha-}$ and $T$ with $\lambda_+^\pm$ replaced by $\lambda_-^\pm$ and $\lambda_{e\pm}$ respectively.
\end{remark}

%%%}}}

\section{Variational principle for eigenvalues of a block operator matrix in a gap of the essential spectrum} %%%{{{
\label{sec:varprinciple}

The aim of this section is to study eigenvalues of an abstract off-diagonal dominant block operator matrix.
In Sections~\ref{sec:DiracRealLine} and \ref{sec:DiracHalfLine} the results will be applied to the Dirac operator on the real line $\R$ and on the halfline $\R_+$.
\smallskip

Let $\mH$ be a Hilbert space and define the block matrix 
\begin{equation*}
   \mT = 
   \begin{pmatrix}
      A & B \\
      B^* & D
   \end{pmatrix},
   \qquad
   \mD(\mT) = \mD(B^*) \oplus \mD(B) \subseteq \mH\oplus\mH
\end{equation*}
where $B$ is a densely defined closed linear operator and $A$ and $D$ are bounded selfadjoint operators on $\mH$.
Then $\mT$ is a selfadjoint linear operator.
\smallskip

We set $d \coloneqq \max \sigma(D)$, $a\coloneqq\min\sigma(A)$ and assume that $d < a $.
Let us define the sesquilinear form
$\mfs(\lambda)$ by
\begin{equation}
\label{eq:SchurformA}
   \mfs(\lambda)[x,y] \coloneqq \scalar{(A-\lambda) x}{y} - \scalar{(D-\lambda)^{-1}B^*x}{B^*y},
   \quad \mD(\mfs(\lambda)) = \mD(B^*),\quad \re \lambda > d.
\end{equation}
We write 
$\mfs(\lambda)[x] \coloneqq \mfs(\lambda)[x,x]$
for the associated quadratic form.
\smallskip

Note that for $\re \lambda > d$ and $x,y\in\mD(B^*)$
\begin{align*}
   \mfs(\lambda)[x,y] 
   = \scalar{(A-\lambda) x}{y} - \scalar{(D-\overline\lambda)(D-\lambda)^{-1}B^*x}{(D-\lambda)^{-1}B^*y},
\end{align*}
hence for $\re \lambda > d$ and $x\in\mD(B^*)$
\begin{align}
   \label{eq:schurAlowerbound}
   \re( \mfs(\lambda)[x] ) &\ge (a-\re\lambda)\|x\|^2 + (\re\lambda-d) \|(D-\lambda)^{-1} B^*x \|^2\ge (a-\re\lambda)\|x\|^2.
\end{align}

Therefore $\mfs(\lambda)$ is a sectorial form and it is not hard to see that it is  closed. 
Hence it defines a closed $m$-sectorial operator $S(\lambda)$ according to \cite[Ch.VI, \S 2]{kato} which has the same lower bound as the form.
Moreover, $\mfs(\lambda)$ is symmetric if $\lambda\in\R$ and the associated operator is selfadjoint in this case.
The operator $S(\lambda)$ is called the \define{Schur complement} of $\mT$.
\smallskip

The next proposition is proved in \cite[Prop. 2.2]{KLT04}.
Our situation corresponds to Case~III there.
\begin{proposition}
   \label{prop:Schurfactorization}
   For $\re\lambda > d$ the operator $\mT-\lambda$ admits the factorization
   \begin{equation}
      \label{eq:SchurA}
      \mT - \lambda
      = \begin{pmatrix}
	 \id & B(D-\lambda)^{-1} \\
	 0 & \id
      \end{pmatrix}
      \begin{pmatrix}
	 S(\lambda) & 0 \\
	 0 & D-\lambda
      \end{pmatrix}
      \begin{pmatrix}
	 \id & 0 \\
	 (D-\lambda)^{-1} B^* & \id 
      \end{pmatrix}.
   \end{equation}
   The $m$-sectorial operator $S(\lambda)$ associated to the form in \eqref{eq:SchurformA} is given by 
   \begin{align*}
      \begin{alignedat}{3}
	 S(\lambda) &= A - \lambda - B(D-\lambda)^{-1}B^*, \\ \mD(S(\lambda)) &= \{ x\in \mD(B^*): (D-\lambda)^{-1}B^*x \in\mD(B)\}, \end{alignedat}
      \qquad & \re\lambda > d.
   \end{align*}
   The pencil $S(\cdot)$ is holomorphic of type (B) (see \cite[Ch.VII, \S 4]{kato}),
   \begin{align*}
      %\label{eq:specSchur}
      \rho(\mT)\cap \{\re\lambda > d \} &= \rho(S)\cap \{\re\lambda > d \}, \\ 
      \sigma(\mT)\cap \{\re\lambda > d \} &= \sigma(S)\cap \{\re\lambda > d \}, \\
      \pointspec(\mT)\cap \{\re\lambda > d \} &= \pointspec(S)\cap \{\re\lambda > d \},
   \end{align*}
   and $\dim \ker(\mT-\lambda) = \dim\ker( S(\lambda))$ for $\re\lambda > d$.
\end{proposition}

\begin{corollary}
   \label{cor:resolventab}
    $(d,a)\subseteq \rho(\mT)$.
\end{corollary}
\begin{proof}
   Let $\lambda\in (d,a)$.
   Then the Schur complement $S(\lambda)$ is defined and 
   for $x\in\mD(S(\lambda)) \subseteq \mD(B^*)$ we have
   \begin{equation*}
      \scalar{S(\lambda)x}{x}
      = \mfs(\lambda)[x] \ge (a-\lambda)\|x\|^2
   \end{equation*}
   by \eqref{eq:schurAlowerbound}.
   Therefore $0\in\rho (S(\lambda))$ and consequently $\lambda\in\rho(\mT)$.
\end{proof}

The next proposition gathers some properties of the family of forms 
$(\mfs(\lambda))_{\re\lambda > d}$ and the operator pencil $(S(\lambda))_{\re\lambda > d}$.
\begin{proposition}\leavevmode
   \label{prop:prepforvariation}
   \begin{enumerate}[label={\upshape (\roman*)}]

      \item
      \label{item:HypothesisVariational:i}
      The domain of the form $\mfs(\lambda)$ is $\mD(B^*)$, in particular it is independent of $\lambda$ for $\re\lambda > d$.

      \item
      \label{item:HypothesisVariational:ii}
      For every $x\in\mD(B^*)$, the function $\mfs(\cdot)[x]$ is continuous. The pencil $S(\cdot)$ is continuous in the norm resolvent sense.

      \item 
      \label{item:HypothesisVariational:iii}
      For every $x\in\mD(B^*) \setminus\{0\}$, the function $(d, \infty)\ni \lambda\mapsto \mfs(\lambda)[x]$ is strictly decreasing.

      \item 
      \label{item:HypothesisVariational:iv}
      For every $\lambda\in (d,a)$, the operator $S(\lambda)$ is strictly positive, in particular its negative spectral subspace is
      $\mL_{(-\infty, 0)} S(\lambda) = \{0\}$.

      \item
      \label{item:HypothesisVariational:v}
      For every $x\in\mD(B^*)\setminus\{0\}$ we have that $\lim_{\lambda\to\infty} \mfs(\lambda)[x] = -\infty$.
      Moreover, for every $\lambda > d$ and $\epsilon>0$ there exists $\delta = \delta(\lambda,\epsilon)$ such that
      $\mfs(\lambda+\epsilon)[x] < 0$ if
      $\mfs(\lambda)[x] < \delta$.

   \end{enumerate}

\end{proposition}
\begin{proof}
   \ref{item:HypothesisVariational:i} follows directly from the definition of $\mfs(\lambda)$ and \ref{item:HypothesisVariational:ii} is a consequence of the continuity of $\lambda\mapsto (D-\lambda)^{-1}$.
   Note that for $d< \mu < \lambda$ and $x\in\mD(B^*)$ we have
   \begin{align*}
      \mfs(\lambda)[x] - \mfs(\mu)[x]
      &= (\mu - \lambda) \|x\|^2 - \scalar{ [ (D-\lambda)^{-1} - (D-\mu)^{-1} ] B^*x}{B^*x}
      \\
      &= (\mu - \lambda) \left( \|x\|^2 + \scalar{ (D-\lambda)^{-1}(D-\mu)^{-1} B^*x}{B^*x} \right)
      \\
      & \le (\mu - \lambda) \left( \|x\|^2 + (d-\lambda)^{-1}(d-\mu)^{-1} \| B^*x\|^2 \right)
      \\
      & \le (\mu - \lambda) \|x\|^2,
   \end{align*}
   which shows \ref{item:HypothesisVariational:iii} and \ref{item:HypothesisVariational:v}.
   The claim \ref{item:HypothesisVariational:iv} follows from the proof of Corollary~\ref{cor:resolventab}.
\end{proof}

Note that Proposition~\ref{prop:prepforvariation} 
implies that the function $\mfs(\cdot)[x]$ has exactly one zero in $(d, \infty)$ if $x\neq 0$.
So for $x\in\mD(B^*)\setminus\{0\}$ we define
\begin{equation*}
    p(x) = \text{unique zero of } \mfs(\cdot)[x] \text{ in } (d, \infty).
\end{equation*}

Observe that $p(x) = p(cx)$ for every $c\in\C\setminus\{0\}$ and that $p(x)\geq a$ by \eqref{eq:schurAlowerbound}.
\smallskip

For $n\in\N$ we define
\begin{equation}
   \label{eq:def:mun}
   \mu_n \coloneqq 
   \min_{ \substack{ L\subseteq \mD(B^*) \\ \dim L = n}}
   \max_{ \substack{ x\in L \\ x\neq 0}} p(x),
   \qquad
   \mu_n' \coloneqq 
   \max_{ \substack{ L\subseteq \mH \\ \dim L = n-1}}
   \inf_{ \substack{ x\in\mD(B^*)\cap L^\perp \\ x\neq 0}} p(x).
\end{equation}

\begin{theorem}
\label{thm:varprinciple:abstract}

    Define 
   \begin{equation*}
      \lambda_{e+} \coloneqq \inf( \essspec(\mT) \cap (d,\infty) ).
   \end{equation*}
    Then $\sigma(\mT)\cap(d, \lambda_{e+}) $ consists of at most countably many eigenvalues of finite multiplicities which may accumulate only at $\lambda_{e+}$.
    Let $(\lambda_n)_{n=1}^\infty$ be the increasing sequence of these eigenvalues counted with multiplicity, extended by $\lambda_{e+}$ if there are only finitely many. Then
    \begin{equation*}
	 \lambda_n = \mu_n = \mu_n',\qquad n\in\N.
    \end{equation*}
\end{theorem}

\begin{proof}
By Proposition~\ref{prop:Schurfactorization} we have that $\lambda_{e+}=\inf( \essspec(S) \cap (d,\infty) )$ and that $(\lambda_n)_{n=1}^\infty$ is also the (possibly extended) sequence of eigenvalues of the pencil $S$ in $(d,\infty)$. The theorem now follows from the variational principle for pencils developed in \cite{EL04}, whose assumptions we have already checked in Proposition~\ref{prop:prepforvariation}.
\end{proof}

It should be noted that 
$\lambda_{e+} = \inf( \essspec(\mT) \cap (d,\infty) )= \inf( \essspec(\mT) \cap (a,\infty) )$
since $(d,a)\subseteq \rho(\mT)$.
In particular
$d < a \leq \lambda_{e+}$.

The formula \eqref{eq:def:mun} now provides a criterion for the existence of eigenvalues in $[a, \lambda_{e+})$.

\begin{corollary}[Existence of eigenvalues]
\label{cor:abstractEVexistencis}

Assume that there exists $x_0\in \mD(B^*)\setminus\{0\}$ and $\lambda\in [a, \lambda_{e+})$ such that
$\mfs(\lambda)[x_0] \le 0$.
Then $\mT$ has at least one eigenvalue in the interval $[a,\lambda]$.
\end{corollary}
\begin{proof}
   Under the assumptions  we obtain that 
   \begin{equation*}
      a \le \mu_1 = 
      \min_{ \substack{ L\subseteq \mD(B^*) \\ \dim L = 1}}
      \max_{ \substack{ x\in L \\ x\neq 0}} p(x)
      \le p(x_0)
      \le \lambda.
      \qedhere
   \end{equation*}
\end{proof}

As an example, let us consider the special case when $A = -D = M\id$ for some constant $M > 0$. This is the case for the Dirac operator with a constant mass term.
\smallskip

The operator $BB^*$ is selfadjoint and non-negative by \cite[Ch.V, Thm.~3.24]{kato}. 
Therefore its spectrum below $\beta_e \coloneqq \inf\essspec(BB^*)$ consists of at most countably many non-negative eigenvalues of finite multiplicity, which may accumulate only at $\beta_e$.

\begin{proposition}

   \label{prop:ADequalMM}
   In the definition of $\mT$, let $A = -D = M\id$ for some constant $M>0$. Let
   $(\beta_n)_{n=1}^\infty$ be the increasing sequence of eigenvalues of $B B^*$ in $[0,\beta_e)$, counted with multiplicity, extended by $\beta_e$ if there are only finitely many.
   Then $(-M, M)\subseteq \rho(\mT)$ and 
    \begin{equation}
      \label{eq:pointspec:M}
      \lambda_{n} = \sqrt{M^2 + \beta_n},\qquad n\in\N.
   \end{equation}
   where $(\lambda_n)_{n=1}^\infty$ is the sequence of eigenvalues of $\mT$  in $[M, \lambda_{e+})$ as defined in Theorem~\ref{thm:varprinciple:abstract}. In particular,
   \begin{equation}
      \label{eq:essspec:M}
      \lambda_{e+} = \sqrt{M^2+\beta_e}.
   \end{equation}
\end{proposition}
\begin{proof}
   That $(-M,M)\subseteq \rho(\mT)$ follows immediately from Corollary~\ref{cor:resolventab} with $a=-d = M$. Let $\lambda \ge M$ and $x\in\mD(B^*)\setminus{0}$, then
   \begin{align*}
      \mfs(\lambda)[x]
      &= \scalar{(A-\lambda)x}{x} - \scalar{(D-\lambda)^{-1} B^*x}{B^* x}
      \\
      &= (M-\lambda) \scalar{x}{x} + (M+\lambda)^{-1} \scalar{B^*x}{B^*x}
      \\
      &= \frac{\|x\|^2}{M+\lambda}  \left( M^2 - \lambda^2 +  \frac{\scalar{B^*x}{B^*x}}{\|x\|^2} \right).
   \end{align*}

   Hence $p(x) = \sqrt{M^2 + \frac{\scalar{B^*x}{B^*x}}{\|x\|^2}}$.
   The min-max principle for semibounded operators \cite[Thm. 12.1]{schmuedgen} shows that 
   \begin{equation*}
      \beta_n = 
      \min_{ \substack{ L\subseteq \mD(B^*) \\ \dim L = n}}
      \max_{ \substack{ x\in L \\ x\neq 0}} 
      \frac{\scalar{B^*x}{B^*x}}{\|x\|^2},
   \end{equation*}
   and from Theorem~\ref{thm:varprinciple:abstract} and \eqref{eq:def:mun} we conclude
\begin{equation*}
   \lambda_n
   = \min_{ \substack{ L\subseteq \mD(B^*) \\ \dim L = n}}
   \max_{ \substack{ x\in L \\ x\neq 0}} 
   \sqrt{ M^2 + \frac{\scalar{B^*x}{B^*x}}{\|x\|^2}}
   = \sqrt{ M^2 + \beta_n}\, .
   \qedhere
\end{equation*}
\end{proof}

\begin{remark}\label{rem:ADequalMM}
   Another way to prove Proposition~\ref{prop:ADequalMM} is to consider the square of $\mT$. Clearly
   \begin{equation*}
      (\mT-\lambda)(\mT+\lambda)
      = \mT^2 - \lambda^2
      = 
      \begin{pmatrix}
	 M^2 + BB^* & 0 \\
	 0 & M^2 + B^*B
      \end{pmatrix}
      -\lambda^2,
      \qquad
      \mD(\mT^2) = \mD(BB^*) \oplus \mD(B^*B),
   \end{equation*}
    where both $BB^*$ and $B^*B$ are non-negative selfadjoint operators by \cite[Ch.V, Thm.~3.24]{kato}. Then we can use 
    the partial isometry
    $$\begin{pmatrix}
    0&B\abs{B}^{-1}\\ B^*\abs{B^*}^{-1} &0    
    \end{pmatrix},$$
    which is unitary when defined from $(\ker B^*)^\perp\oplus(\ker B)^\perp$ to itself, to show that
    \begin{align*}
        \sigma(BB^*)\setminus\{0\} &= \sigma(B^*B)\setminus\{0\},\\
        \sigma(\mT)\setminus\{\pm M\} &= \sigma(-\mT)\setminus\{\pm M\},
    \end{align*}
    with the same holding for $\pointspec$ with equal multiplicity (see \cite[Sec. 5.2.3]{thaller}). 
    These results were obtained particularly for the Dirac operator $T$ in \cite{JakubskyKrejcirik2014} .
\end{remark}

%%%}}}

\section{The Dirac operator on the real line} %%%{{{
\label{sec:DiracRealLine}

In this section we use Theorem~\ref{thm:varprinciple:abstract} to give sufficient conditions for the existence of eigenvalues of the operator $T$ defined in Proposition~\ref{prop:defT}.
Recall that 
\begin{equation*}
   T = \begin{pmatrix}
      0 & \frac{\rd}{\rd x}  \\ -\frac{\rd}{\rd x} & 0
   \end{pmatrix}
   + V
   =
   \begin{pmatrix}
      0 & \frac{\rd}{\rd x}  \\ -\frac{\rd}{\rd x} & 0
   \end{pmatrix}
   +
   \begin{pmatrix}
      M_1 & W \\ W & -M_2
   \end{pmatrix}
   =
   \begin{pmatrix}
      M_1 & B \\ B^* & -M_2
   \end{pmatrix},\quad\mD(T)=H^1(\R)^2,
\end{equation*}
where $B= \frac{\rd}{\rd x } + W$, $B^*= -\frac{\rd}{\rd x } + W$, $\mD(B)=\mD(B^*)=H^1(\R)$ and the functions $M_1,M_2,W$ satisfy \ref{item:A1}, \ref{item:A2} \ref{item:A3}. Recall as well the definition of $\lambda_{e\pm}$ and $m_i$ given in \eqref{eq:def:lambda_e} and \ref{item:A2} respectively.
\smallskip

Clearly, $T$ satisfies all the hypotheses of the previous section.
The following proposition sums up the information that we have so far on the essential spectrum and resolvent set of $T$.

\begin{proposition}
   \label{prop:EssspecResT}
   Under the assumptions above, 
   \begin{equation*}
      \essspec(T) = (-\infty, \lambda_{e-}] \cup [\lambda_{e+}, \infty),
      \qquad
      (-m_2, m_1)\subseteq \rho(T).
   \end{equation*}
   A necessary condition for the existence of discrete spectrum of $T$ is that 
   $\lambda_{e-} < -m_2$ or $m_1 < \lambda_{e+}$.
\end{proposition}

As an immediate consequence we obtain that for constant $M_j$ and $W$ vanishing at infinity the only possible eigenvalues of $T$ in $[\lambda_{e-}, \lambda_{e+}]=[-M_2,M_1]$ are $-M_2$ and $M_1$, which then clearly are not discrete eigenvalues.
If the conditions of Remark~\ref{rem:AbsCont} are satisfied, then $T$ has no embedded eigenvalues in the essential spectrum and therefore $-M_2$ and $M_1$ are the only possible eigenvalues of $T$.
\smallskip

The family of quadratic forms \eqref{eq:SchurformA} associated with $T$ for $\re(\lambda) > -m_2$ is
\begin{align*}
   \mfs(\lambda)[\phi] 
   &= \scalar{(M_1-\lambda) \phi}{\phi} + \scalar{(M_2+\lambda)^{-1}B^*\phi}{B^*\phi}
   \\
   &= \int_\R (M_1-\lambda) \abs{\phi}^2 + \frac{\abs{-\phi' + W\phi}^2}{M_2+\lambda}\,\rd x
   \\
   & =\int_\R \frac{-\det(V-\lambda)\abs{\phi}^2 + \abs{\phi'}^2- 2W\re(\phi\phi')}{M_2+\lambda}\,\rd x
\end{align*}
with $\mD(\mfs(\lambda)) =\mD(B^*) = H^1(\R)$.
Recall that the eigenvalues $\lambda_n$ of $T$ in $[m_1, \lambda_{e+})$ are simple and given by
\begin{equation*}
   \mu_n = 
   \min_{ \substack{ L\subseteq \mD(B^*) \\ \dim L = n}}
   \max_{ \substack{ \phi\in L \\ \phi\neq 0}}\ p(\phi),
   \qquad
   \mu_n' =
   \max_{ \substack{ L\subseteq \mH \\ \dim L = n-1}}
   \min_{ \substack{ \phi\in\mD(B^*)\cap L^\perp \\ \phi\neq 0}} p(\phi),
\end{equation*}
as long as $\mu_n < \lambda_{e+}$.
\smallskip

We now give a series of theorems about the existence, number and location of the eigenvalues of $T$ in $[m_1, \lambda_{e+})$. 
Analogous results can be obtained for the eigenvalues of $T$ in $(\lambda_{e-},-m_2]$ by considering the operator $\widetilde{T}$ and the unitary transformation given by
\begin{equation}
\label{eq:Ttransformed}
    \widetilde{T}\coloneqq\begin{pmatrix}
      M_2 &  \frac{\rd}{\rd x} - W \\
      -\frac{\rd}{\rd x} - W & -M_1
   \end{pmatrix},\qquad\begin{pmatrix}
     0 &  \id \\
      \id & 0
   \end{pmatrix}
   \widetilde{T}
   \begin{pmatrix}
     0 &  \id \\
      \id & 0
   \end{pmatrix}^{*} = -T.
\end{equation}
Alternatively, one could also use the factorization, analogous to \eqref{eq:SchurA},
\begin{equation}
\label{eq:SchurD}
    T - \lambda = 
    \begin{pmatrix}
        \id & 0 \\ B^*(M_1-\lambda)^{-1} & \id
    \end{pmatrix}
    \begin{pmatrix}
        M_1-\lambda & 0 \\ 0 & \widetilde S(\lambda)
    \end{pmatrix}
    \begin{pmatrix}
        \id & (M_1-\lambda)^{-1}B \\ 0 & \id 
    \end{pmatrix}
\end{equation}
with the Schur complement
\begin{align*}
   \begin{alignedat}{3}
    \widetilde S(\lambda) &= -M_2 -\lambda - B^*(M_1-\lambda)^{-1}B, \\ \mD(\widetilde S(\lambda)) &= \{ x\in \mD(B): (M_1-\lambda)^{-1}Bx \in\mD(B^*)\}, \end{alignedat}
      \qquad & \re\lambda < m_1.
   \end{align*}

We start with two theorems on the existence of at least one eigenvalue in $[m_1,\lambda_{e+})$. These are extensions of the Theorem presented in \cite[Sec.~2.2]{JakubskyKrejcirik2014}. 
We recover their result by setting $M_1(x)=M_2(x)=M>0$.
\begin{theorem}
   %%%{{{
   \label{thm:IntegrabilityCond}
      If
      $\left(\dfrac{ -\det(V-\lambda_{e+}) }{M_2+\lambda_{e+}}\right)^+\in L_1(\R)$ and if
      \begin{equation}
	 \label{eq:assNonnegative}
        \int_\R \dfrac{ -\det(V-\lambda_{e+})}{M_2+\lambda_{e+}}\, \rd x
      - \left(\dfrac{W_-}{\Mlim_{2-} + \lambda_{e+}}-\dfrac{W_+}{\Mlim_{2+} + \lambda_{e+}}\right)
      <0,
      \end{equation}
      then $T$ has at least one eigenvalue in $[m_1,\lambda_{e+})$.
\end{theorem}
\begin{remark}
   \label{rem:impossible}
    The condition $\lambda_-^+=\lambda_+^+$ is
    a necessary for $\left(\dfrac{ -\det(V-\lambda_{e+}) }{M_2+\lambda_{e+}}\right)^+\in L_1(\R)$.
    
    To see this, recall that $\lambda_{e+} = \min\{\lambda_+^+,\, \lambda_-^+\}$ and notice that the polynomial $\lambda\mapsto-\det(V_\pm-\lambda)$ is larger than $0$ for $\lambda\in (\lambda_\pm^-,\, \lambda_\pm^+)$.
    If $\lambda_-^+ < \lambda_+^+$ then $\lambda_{e+}=\lambda_-^+$ and therefore
    $\lambda_-^+ \in (\lambda_+^-,\, \lambda_+^+)$ by \eqref{eq:lambdainequalities}.
    Hence
    \begin{equation*}
        \lim_{x\to+\infty}\dfrac{ -\det(V-\lambda_{e+})}{M_2+\lambda_{e+}}= \dfrac{  -\det(V_+-\lambda_-^+)}{\Mlim_{2+}+\lambda_-^+}>0,
    \end{equation*}
    so we cannot have integrability at $+\infty$.
    If $\lambda_+^+ < \lambda_-^+$ then a similar argument shows that 
    \begin{equation*}
        \lim_{x\to-\infty}\dfrac{ -\det(V-\lambda_{e+})}{M_2+\lambda_{e+}}= \dfrac{ -\det(V_--\lambda_+^+)}{\Mlim_{2-}+\lambda_+^+}>0,
    \end{equation*}
    again showing non-integrability.
\end{remark}
%%%}}}

\begin{proof}[Proof of Theorem~\ref{thm:IntegrabilityCond}]
%%%{{{
   We will show that there exists a function $\phi\in H^1(\R)$ such that 
   $\mfs(\lambda_{e+})[\phi] < 0$.
   Then, by Corollary~\ref{cor:abstractEVexistencis}, the operator $T$ has an eigenvalue in $[m_1, \lambda_{e+})$.
   To this end, we define the sequence of piecewise linear functions
   \begin{equation*}
      \phi_n(x)\coloneqq
      \begin{cases}
	 0, & x \in (-\infty , -2n] \cup [2n , \infty),\\[1ex]
	  \dfrac{2n - |x| }{n} , & x \in (-2n , -n) \cup (n , 2n),\\[1ex]
	 1, & x \in [-n , n].
      \end{cases}
   \end{equation*}
   Clearly, $\phi_n\in H^1(\R)$ and 
   \begin{equation}
      \label{eq:forminlambda}
      \mfs(\lambda_{e+})[\phi_n]
      = \int_\R \dfrac{ -\det(V-\lambda_{e+})\phi_n^2}{M_2+\lambda_{e+}}\, \rd x
      +\int_\R\dfrac{ (\phi_n')^2}{M_2+\lambda_{e+}}\, \rd x
      -2\int_\R\dfrac{ W\phi_n \phi_n' }{M_2+\lambda_{e+}}\, \rd x . 
   \end{equation}
   The second and third terms on the right hand side of \eqref{eq:forminlambda} converge since
   \begin{align*}
      \int_\R \dfrac{(\phi_n')^2}{M_2+\lambda_{e+}}\, \rd x 
      & = \dfrac{1}{n^2}\left[\int_{-2n}^{-n} \dfrac{1}{M_2+\lambda_{e+}}\, \rd x
      +\int_{n}^{2n} \dfrac{1}{M_2+\lambda_{e+}}\rd x\right]
      \leq\dfrac{2}{ n (m_2+\lambda_{e+})} 
      \ \xrightarrow{n\to\infty}\ 0,
      \\[2ex]
      \int_\R \dfrac{ W\phi_n \phi_n' }{M_2+\lambda_{e+}}\, \rd x 
      & = \int_{-2n}^{-n} \dfrac{(x+2n)W}{n^2(M_2+\lambda_{e+})}\, \rd x 
      - \int_{n}^{2n} \dfrac{(2n-x)W}{n^2(M_2+\lambda_{e+})}\, \rd x
      \\ 
      &\hspace*{6cm} \xrightarrow{n\to\infty}\ 
      \frac{1}{2} \left(\dfrac{W_-}{\Mlim_{2-} + \lambda_{e+}}-\dfrac{W_+}{ \Mlim_{2+} + \lambda_{e+} }\right).
   \end{align*}
   Assumption \eqref{eq:assNonnegative} implies that the first term on the right hand side of \eqref{eq:forminlambda} also converges. By splitting the integrand into its positive and negative parts and then applying the monotone convergence theorem to both of them, to find that
   \begin{equation*}
      \int_\R \dfrac{  -\det(V-\lambda_{e+})\phi_n^2 }{M_2 + \lambda_{e+}}\, \rd x 
      \ \xrightarrow{n\to\infty} \
      \int_\R \dfrac{ -\det(V-\lambda_{e+})}{M_2 + \lambda_{e+}}\, \rd x.
   \end{equation*}
   Therefore
   \begin{equation*}
      \lim_{n\to\infty}\mfs(\lambda_{e+})[\phi_n]
      = \int_\R \dfrac{ -\det(V-\lambda_{e+})}{M_2+\lambda_{e+}}\, \rd x
      - \left(\dfrac{W_-}{\Mlim_{2-} + \lambda_{e+}}-\dfrac{W_+}{\Mlim_{2+} + \lambda_{e+}}\right)
      <0
   \end{equation*}
   as desired.
\end{proof}
%%%}}}

\begin{theorem} %%%{{{
   \label{thm:aEqualsEssspec}
    Suppose that $M_1, M_2, W$ are differentiable with bounded derivatives and that there exists $a\in\R$ such that 
    $\det(V(a)-\lambda_{e+}) = 0$. The following two sets of conditions are sufficient for the existence of an eigenvalue of $T$ in $[m_1,\, \lambda_{e+})$.
   \begin{enumerate}[label={\upshape (\roman*)}]
    \item $\left(\dfrac{ -\det(V-\lambda_{e+})}{ M_2 + \lambda_{e+}}\right)^+\in L_1(-\infty,a)$ and
    \begin{multline*}
	 \int_{-\infty}^a \dfrac{ -\det(V-\lambda_{e+})}{ M_2 + \lambda_{e+}}\, \rd x 
      -\dfrac{W_-}{ \Mlim_{2-} + \lambda_{e+} }
      +\sup_{x\in[a,\infty)} \dfrac{W}{ M_2 + \lambda_{e+} } \\
      + \frac{1}{2} 
      \left(\frac{9}{2} \right)^{1/3}
      \left[\sup_{x\in [a, \infty)}  
      \left( \dfrac{ -\det(V-\lambda_{e+}) }{ M_2 + \lambda_{e+}} \right)' 
      \right]^{1/3}
      \left[\sup_{x\in[a,\infty)}\dfrac{1}{M_2 + \lambda_{e+}}\right]^{2/3}<0.
      \end{multline*}
       \item $\left(\dfrac{ -\det(V-\lambda_{e+})}{ M_2 + \lambda_{e+}}\right)^+\in L_1(a,\infty)$ and
      \begin{multline*}
	 \int_{a}^\infty \dfrac{ -\det(V-\lambda_{e+})}{ M_2 + \lambda_{e+}}\, \rd x 
      +\dfrac{W_+}{ \Mlim_{2+} + \lambda_{e+} }
      +\sup_{x\in(-\infty,a]} \dfrac{-W}{ M_2 + \lambda_{e+} } \\
      + \frac{1}{2} 
      \left(\frac{9}{2} \right)^{1/3}
      \left[\sup_{x\in (-\infty,a]}  
      \left( \dfrac{ \det(V-\lambda_{e+}) }{ M_2 + \lambda_{e+}} \right)' 
      \right]^{1/3}
      \left[\sup_{x\in(-\infty,a]}\dfrac{1}{M_2 + \lambda_{e+}}\right]^{2/3}<0.
      \end{multline*}
   \end{enumerate}
\end{theorem}
%%%}}}
\begin{remark}
    As in Remark~\ref{rem:impossible}, $\lambda_-^+\leq \lambda_+^+$ is necessary for (i), and $\lambda_+^+\leq \lambda_-^+$ is necessary for (ii).
\end{remark}
\begin{proof}[Proof of Theorem~\ref{thm:aEqualsEssspec}]
    %%%{{{
   We prove only (i), since the proof of (ii) is analogous. 
   For $b>a$ we define $\phi_n\in H^1(\R)$ by
   \begin{equation*}
      \phi_n(x)\coloneqq
      \begin{cases}
	 0, & x \in (-\infty , -2n] \cup [b , \infty),\\[1ex]
	 \dfrac{x+2n}{n}, & x \in (-2n , -n),\\[1ex]
	 1, & x \in [-n , a],\\[1ex]
	 \dfrac{b-x}{b-a}, & x \in (a , b).
      \end{cases}
   \end{equation*}
    We split $\mfs(\lambda_{e+})[\phi_n]$ into three terms as in \eqref{eq:forminlambda} and take the limits separately. For the last two terms we have
   \begin{align*}
   \lim_{n\to\infty} \int_\R \dfrac{(\phi_n')^2}{M_2 + \lambda_{e+}}\, \rd x
    &= \dfrac{1}{(b-a)^2}\int_a^b \dfrac{1}{M_2 + \lambda_{e+}}\, \rd x
    \leq \dfrac{1}{(b-a)}\sup_{x\in[a,\infty)}\dfrac{1}{M_2 + \lambda_{e+}},\\[2ex]
      \lim_{n\to\infty} -2\int_\R \dfrac{ W\phi_n \phi_n' }{ M_2 + \lambda_{e+} }\, \rd x
      &= -\dfrac{W_-}{ \Mlim_{2-} + \lambda_{e+} } + \dfrac{2}{b-a}\int_a^b \dfrac{W}{ M_2 + \lambda_{e+} } \dfrac{b-x}{b-a}\, \rd x\\
      &\leq -\dfrac{W_-}{ \Mlim_{2-} + \lambda_{e+} }
      +\sup_{x\in[a,\infty)} \dfrac{W}{ M_2 + \lambda_{e+} }.
   \end{align*}
For the first term we have  
    \begin{multline*}
      \lim_{n\to\infty} \int_\R \dfrac{  -\det(V-\lambda_{e+}) \phi_n^2 }{ M_2 + \lambda_{e+}}\, \rd x
      \\
      \begin{aligned}[t]
	 & = \lim_{n\to\infty} \int_{-\infty}^a \dfrac{ -\det(V-\lambda_{e+}) \phi_n^2 }{ M_2 + \lambda_{e+}}\, \rd x
	 +
	 \int_{a}^b \dfrac{  -\det(V-\lambda_{e+}) }{ M_2 + \lambda_{e+}}\left(\frac{b-x}{b-a}\right)^2\, \rd x
	 \\
	 & = \int_{-\infty}^a \dfrac{  -\det(V-\lambda_{e+}) }{ M_2 + \lambda_{e+}}\, \rd x
	 +
	 \int_{a}^b \left( \dfrac{  -\det(V-\lambda_{e+})}{ M_2 + \lambda_{e+}} \right)' \frac{ (b-x)^3 }{ 3(b-a)^2 } \, \rd x
	 \\
	 & \leq \int_{-\infty}^a \dfrac{ -\det(V-\lambda_{e+})}{ M_2 + \lambda_{e+}}\, \rd x
	 +
	 \frac{ (b-a)^2 }{ 12 }
	 \sup_{x\in [a, \infty)} \left( \dfrac{ -\det(V-\lambda_{e+})}{ M_2 + \lambda_{e+}} \right)' 
      \end{aligned}
   \end{multline*}
   where we used the assumption
   $\det(V(a)-\lambda_{e+}) = 0$ when we performed the integration by parts.
   \smallskip

   Hence
   \begin{align*}
      \lim_{n\to\infty}\mfs(\lambda_{e+})[\phi_n]
      & \leq 
      \int_{-\infty}^a \dfrac{ -\det(V-\lambda_{e+})}{ M_2 + \lambda_{e+}}\, \rd x 
      -\dfrac{W_-}{ \Mlim_{2-} + \lambda_{e+} }
      +\sup_{x\in[a,\infty)} \dfrac{W}{ M_2 + \lambda_{e+} } 
      \\
      & \phantom{=\ }+ \frac{ (b-a)^2 }{ 12 } \sup_{x\in [a, \infty)} 
      \left( \dfrac{ -\det(V-\lambda_{e+})}{ M_2 + \lambda_{e+}} \right)'
      +\dfrac{1}{(b-a)}\sup_{x\in[a,\infty)}\dfrac{1}{M_2 + \lambda_{e+}}
      .
   \end{align*}
   Note that $\displaystyle\sup_{x\in[a,\infty)}\dfrac{1}{M_2+\lambda_{e+}}>0$ 
   and 
   $\displaystyle\sup_{x\in[a,\infty)}
   \left( \dfrac{ -\det(V-\lambda_{e+})}{ M_2 + \lambda_{e+}} \right)' 
   \geq0$, 
   therefore
   \begin{equation*}
      \frac{ (b-a)^2 }{ 12 } \sup_{x\in [a, \infty)} 
      \left( \dfrac{ -\det(V-\lambda_{e+})}{ M_2 + \lambda_{e+}} \right)'
      +\dfrac{1}{(b-a)}\sup_{x\in[a,\infty)}\dfrac{1}{M_2 + \lambda_{e+}}
   \end{equation*}
   is minimized at
   $b_0 = a + \left[\sup_{x\in[a,\infty)}\dfrac{6}{M_2+\lambda_{e+}}\right]^{1/3}\left[\sup_{x\in[a,\infty)}
   \left( \dfrac{ -\det(V-\lambda_{e+})}{ M_2 + \lambda_{e+}} \right)'\right]^{-1/3}$.
   \\
   Evaluation at $b_0$ yields
   \begin{align*}
      \lim_{n\to\infty}\mfs(\lambda_{e+})[\phi_n]
      & \leq 
      \int_{-\infty}^a \dfrac{ -\det(V-\lambda_{e+})}{ M_2 + \lambda_{e+}}\, \rd x 
      -\dfrac{W_-}{ \Mlim_{2-} + \lambda_{e+} }
      +\sup_{x\in[a,\infty)} \dfrac{W}{ M_2 + \lambda_{e+} } \\
      & \phantom{=\; \ }
      + \frac{1}{2} 
      \left(\frac{9}{2} \right)^{1/3}
      \left[\sup_{x\in [a, \infty)}  
      \left( \dfrac{ -\det(V-\lambda_{e+}) }{ M_2 + \lambda_{e+}} \right)' 
      \right]^{1/3}
      \left[\sup_{x\in[a,\infty)}\dfrac{1}{M_2 + \lambda_{e+}}\right]^{2/3}<0.
   \end{align*}
\end{proof}
%%%}}}

Next we compare the eigenvalues of the Dirac operator $T$ with those of Sturm-Liouville operators constructed from the operator $B=\frac{\rd}{\rd x}+W$ and its adjoint, or their restrictions to an interval. For an open interval $I$, we define the quantities
\begin{equation}\label{eq:def:m&M}
    m_{i,I}\coloneqq\inf_{x\in I}M_i(x),\qquad \Msup_{i,I}\coloneqq\sup_{x\in I}M_i(x),\qquad i=1,2,
\end{equation}
and the functions
\begin{align}
   f_I(\lambda) &\coloneqq ( \lambda-\Msup_{1,I})(m_{2,I} + \lambda),\qquad\lambda\in[\Msup_{1,I},\infty),\label{eq:def:f}
   \\
   g_I(\lambda) &\coloneqq (\lambda-m_{1,I})(\Msup_{2,I} + \lambda),\qquad\lambda\in[m_{1,I},\infty), \label{eq:def:g}
\end{align}
and remark that $f_I,g_I$ are both strictly increasing bijections onto $[0,\infty)$. When $I=\R$ we will omit the subscript $I$. Note that $m_{i,\R} = m_i$ as defined in \ref{item:A2}.
\smallskip

From the case $M_1(x)=M_2(x)=M>0$, we can extract information about the spectrum of the operator $BB^*$ as follows. 
Proposition \ref{prop:ADequalMM}, Remark \ref{rem:ADequalMM} and \eqref{eq:def:lambda_e} easily imply that $\essspec(BB^*)=[\beta_e,\infty)$ where $\beta_e=\inf \essspec(BB^*)=\min\{ W_-^2,\, W_+^2\}$. 
In addition, since every eigenvalue of $T$ is simple, as shown in Proposition \ref{prop:defT}, the same is true for the eigenvalues of $BB^*$ in $[0,\beta_e)$.
\smallskip

The estimates in the next theorem for the eigenvalues of $T$ have already been shown by Kraus, Langer and Tretter in \cite[Thm. 4.7]{KLT04}. In our case however we can simplify their proof slightly.
In the special case when $M_1(x)=M_2(x)=M>0$, we recover the results from Proposition~\ref{prop:ADequalMM}.

\begin{theorem}
\label{thm:teoremafavoritoglobal}
   Let $(\lambda_n)_{n=1}^\infty$ be the eigenvalues of $T$ in $[m_1, \lambda_{e+})$ and 
   $(\beta_n)_{n=1}^\infty$ be the eigenvalues of $BB^*$ in $[0, \beta_e)$.
    \begin{enumerate}[label={\upshape (\roman*)}]
    
    \item $0\leq\beta_e\leq g(\lambda_{e+})$ and for $n=1,\ldots, \dim\mL_{[0,\beta_e)}BB^*$ we have
    $$\lambda_n\geq g^{-1}(\beta_n)= \frac{ m_1 - \Msup_2 }{2}  
      + \sqrt{\left( \frac{ m_1 + \Msup_2 }{2} \right)^2 +  \beta_n }\ .$$
      In particular, $\dim \mL_{[m_1, \lambda)}T
      \leq\dim \mL_{[0,g(\lambda))}BB^*$ for any $\lambda\in[m_1, g^{-1}(\beta_e)]$.
      
      \item If $\Msup_1 \leq \lambda_{e+}$, then $0\leq f(\lambda_{e+})\leq\beta_e$ and for $n=1,\ldots, \dim\mL_{[0,f(\lambda_{e+}))}BB^*$ we have
      $$\lambda_n\leq f^{-1}(\beta_n)=  \frac{ \Msup_1 - m_2 }{2}  
      + \sqrt{  \left( \frac{\Msup_1 + m_2 }{2} \right)^2 +\beta_n}\ .$$
      In particular, 
      $\dim \mL_{[m_1, \lambda)}T\geq\dim \mL_{[0,f(\lambda))}BB^*$ for any $\lambda\in[m_1, \lambda_{e+}]$, where we use the convention $f(\lambda)=0$ for $\lambda\leq\Msup_1$.
    \end{enumerate}
\end{theorem}
%%%}}}

\begin{proof}\leavevmode
\begin{enumerate}[label={\upshape (\roman*)}]
    \item Notice that $g^{-1}(\beta)=\eta(m_1,\Msup_2,\beta)$ where
    $$\eta(x,y,z)\coloneqq\frac{x-y}{2}+\sqrt{\left(\frac{x+y}{2}\right)^2+z}\ .$$
    Differentiation shows that, on the set $\R^2\times\R_+$, $\eta$ is strictly increasing in $x,z$ and strictly decreasing in $y$. Therefore we have
    \begin{align*}
        \lambda_{e+}=\min\{\lambda^+_+,\, \lambda^+_-\}
        &=\min\left\{\eta(\Mlim_{1+},\Mlim_{2+},W_{+}^2),\ \eta(\Mlim_{1-},\Mlim_{2-},W_{-}^2)\right\}\\
        &\geq \min\left\{\eta(m_{1},\Msup_{2},W_{+}^2),\ \eta(m_{1},\Msup_{2},W_{-}^2)\right\}\\
        &=\eta(m_{1},\Msup_{2},\min\{W_{\pm}^2\})\\
        &=\eta(m_{1},\Msup_{2},\beta_e)
        =g^{-1}(\beta_e)\geq m_1,
    \end{align*}
    and, after applying $g$, $0\leq\beta_e\leq g(\lambda_{e+})$.
    
    Let $\phi_1, \dots, \phi_n$ be eigenfunctions of $BB^*$ with eigenvalues $0\leq \beta_1 < \beta_2 < \cdots <\beta_n < \beta_e$. Then there exist unique 
   $m_1 \le \widetilde \lambda_1 < \widetilde \lambda_2 < \cdots < \widetilde \lambda_n <g^{-1}(\beta_e)\leq\lambda_{e+}$ 
   such that $g(\widetilde\lambda_j) = \beta_j$.
   
   Let $L\subseteq \mD(B^*)$ with $\dim L = n$.
   Then we can choose a normalized $\psi_L\in L$ such that $\psi_L\perp \{ \phi_1, \dots, \phi_{n-1}\}$,
   in particular $\hp{B^*\psi_L}{B^*\psi_L}\geq\beta_n$. With these observations we have
   \begin{align*}
      \mfs(\lambda)[\psi_L]
      &= \int_\R (M_1 - \lambda) \abs{\psi_L}^2\, \rd x
      +\int_\R \frac{\abs{ -\psi_L' + W \psi_L }^2}{ M_2 + \lambda}\, \rd x
      \\
      &\geq \left(m_1 - \lambda\right) \int_\R\abs{\psi_L}^2\, \rd x
      +\frac{1}{ \Msup_2 + \lambda} \int_\R\abs{ -\psi_L' + W \psi_L}^2\, \rd x
      \\
      &= (m_1 - \lambda) + \frac{\scalar{B^*\psi_L}{ B^*\psi_L}}{ \Msup_2 + \lambda }\\
      &= \frac{ -g (\lambda) + \scalar{B^* \psi_L}{B^* \psi_L}}{  \Msup_2 + \lambda }\\
      &\geq \frac{ -g (\lambda) + \beta_n}{  \Msup_2 + \lambda },
   \end{align*}
    which implies $p(\psi_L)\geq\widetilde\lambda_n$. Since $L$ is arbitrary, it follows that
\begin{equation*}
    \lambda_{n} = 
      \min_{ \substack{ L\subseteq \mD(B^*) \\ \dim L = n} }
      \max_{ \substack{ \phi\in L \\ \phi\neq 0} }\ p(\phi)
      \ge
      \min_{ \substack{ L\subseteq \mD(B^*) \\ \dim L = n} }
      p(\psi_L)
      \ge \widetilde \lambda_{n}=g^{-1}(\beta_n).
\end{equation*} 

\item Notice that $f^{-1}(\beta)=\eta(\Msup_1,m_2,\beta)$. 
Therefore we have
\begin{align*}
    \Msup_1\leq \lambda_{e+}
    =\min\{\lambda^+_+,\ \lambda^+_-\}
    &=\min\left\{\eta(\Mlim_{1+},\Mlim_{2+},W_{+}^2),\ \eta(\Mlim_{1-},\Mlim_{2-},W_{-}^2)\right\}\\ 
    &\leq \min\left\{\eta(\Msup_{1},m_2,W_{+}^2),\ \eta(\Msup_{1},m_2,W_{-}^2)\right\}\\
    &=\eta(\Msup_{1},m_2,\min\{W_{\pm}^2\})\\
    &=\eta(\Msup_{1},m_2,\beta_e)
    =f^{-1}(\beta_e),
    \end{align*}
    and, after applying $f$, $0\leq f(\lambda_{e+})\leq \beta_e$.

   Let $\phi_1, \dots, \phi_n$ be eigenfunctions of $BB^*$ with eigenvalues $0\leq \beta_1 < \beta_2 < \cdots <\beta_n < f(\lambda_{e+})$.
   Then there exist unique 
   $\Msup_1 \le \widetilde \lambda_1 < \widetilde \lambda_2 < \cdots < \widetilde \lambda_n < \lambda_{e+}\leq f^{-1}(\beta_e)$ 
   such that $f(\widetilde\lambda_j) = \beta_j$.
   Hence if $\phi\in\linspan\{ \phi_1,\,\dots,\, \phi_n\}\subseteq\mD(B^*)$ is normalized, then
   \begin{align*}
      \mfs(\lambda)[\phi]
      &= \int_\R (M_1-\lambda) \abs{\phi}^2\, \rd x
      +\int_\R \frac{\abs{ -\phi' + W \phi }^2}{ M_2 + \lambda}\, \rd x
      \\
      &\leq \left(\Msup_1 - \lambda\right)\int_\R\abs{\phi}^2\, \rd x
      +\frac{1}{ m_2 + \lambda} \int_\R\abs{ -\phi' + W \phi}^2\, \rd x
      \\
      &= (\Msup_1 - \lambda) + \frac{\scalar{B^*\phi}{ B^*\phi}}{ m_2 + \lambda}\\
      &= \frac{ -f (\lambda) + \scalar{B^* \phi}{B^* \phi}}{ m_2 + \lambda}\\
      &\leq \frac{ -f (\lambda) + \beta_n}{ m_2 + \lambda},
   \end{align*}
   and therefore $p(\phi)\leq \widetilde\lambda_n$. The proof is finished by
   \begin{equation*}
       \lambda_n = 
      \min_{ \substack{ L\subseteq \mD(B^*) \\ \dim L = n} }
      \max_{ \substack{ \phi\in L \\ \phi\neq 0} }\ p(\phi)
      \leq \max_{ \substack{ \phi\in \linspan\{ \phi_1,\,\dots,\, \phi_n\} \\ \phi\neq 0} } p(\phi) \leq \widetilde \lambda_n=f^{-1}(\beta_n).\qedhere
   \end{equation*}
\end{enumerate}
\end{proof}

\begin{corollary}
    From Theorem~\ref{thm:teoremafavoritoglobal} (i) it follows immediately that $[m_1,g^{-1}(\inf\sigma(BB^*))\subseteq\rho(T)$.
\end{corollary}

The next goal is to prove a localized version of the above theorem which allows us to compare the eigenvalues of $T$ with those of Sturm-Liouville operators on intervals.

Let $I$ be a bounded open interval and define the differential operator $B_I$ and its adjoint by
\begin{align*}
     B_I &= \frac{\rd}{\rd x} + W,\qquad \mD(B_I)  = H^1(I),\\
     B_I^* &= -\frac{\rd}{\rd x} + W,\qquad \mD(B_I^*)  = H_0^1(I),
\end{align*}
in the Hilbert space $L_2(I)$,
see \cite[Ch.V, Ex.~3.25]{kato}.

The second order differential operator $B_IB_I^*$ is selfadjoint and non-negative \cite[Ch.V, Thm.~3.24]{kato}. 
Setting $u(x)\coloneqq\exp\left(\int_0^x W(s)\rd s\right)$, the unitary map
\begin{align*}
    U:L_2(I)&\longrightarrow L_2(I,u^2\rd x)\\
    \phi&\longmapsto \phi/u
\end{align*}
transforms $B_IB^*_I$ into a standard Sturm-Liouville operator
$$UB_IB^*_IU^*=-\frac{1}{u^2}\frac{\rd}{\rd x}u^2\frac{\rd}{\rd x}$$
defined on the Hilbert space $L_2(I,u^2\rd x)$ with Dirichlet boundary conditions. 
By the spectral theory of Sturm-Liouville operators on a bounded interval, we conclude that $B_IB_I^*$ has purely discrete spectrum, all its eigenvalues are simple, see \cite[Satz 13.21]{weidmannII}, and by the min-max principle \cite[Thm. 12.1]{schmuedgen} they are
\begin{align*}
   \beta_{n,I} = 
   \min_{ \substack{ L\subseteq \mD(B^*_I) \\ \dim L = n}}
   \max_{ \substack{ x\in L \\ x\neq 0}} 
   \frac{\scalar{B_I^*x}{B_I^*x}}{\|x\|^2}.
\end{align*}

\begin{theorem} %%%{{{
   \label{thm:teoremafavorito}
    Let $(\lambda_n)_{n=1}^\infty$ be the eigenvalues of $T$ in $[m_1, \lambda_{e+})$ and 
   $(\beta_{n,I})_{n=1}^\infty$ the eigenvalues of $B_IB_I^*$ for some open bounded open interval $I$.\smallskip

   If $\Msup_{1,I} \leq \lambda_{e+}$, then for $n=1,\ldots, \dim\mL_{[0,f_I(\lambda_{e+}))}B_IB_I^*$ we have
   \begin{equation}
      \label{eq:upperboundInterval} 
      \lambda_n\leq f_I^{-1}(\beta_{n,I})=  \frac{ \Msup_{1,I} - m_{2,I} }{2}  
      + \sqrt{  \left( \frac{\Msup_{1,I} + m_{2,I} }{2} \right)^2 +\beta_{n,I}}\ .
   \end{equation}
   In particular, $\dim \mL_{[m_1, \lambda)}T\geq\dim \mL_{[0,f_I(\lambda))}B_IB_I^*$ for any $\lambda\in[m_1, \lambda_{e+}]$, where we use the convention $f_I(\lambda)=0$ for $\lambda\leq\Msup_{1,I}$.
\end{theorem}
%%%}}}
\begin{proof} %%%{{{
   The proof is the same as the one of Theorem~\ref{thm:teoremafavoritoglobal} (ii). We just need to identify functions in $\mD(B_IB_I^*)\subseteq\mD(B_I^*)=H_0^1(I)$ with functions in $\mD(B^*)=H^1(\R)$ by extending them by $0$ outside of $I$. 
\end{proof}

Since in the above Theorem only the local behaviour of the $M_j$ and $W$ is taken into account, it can give only a lower bound for the total number of eigenvalues in the gap and upper bounds for their values.
\smallskip

The theorem can be extended in a natural way to the situation of several disjoint intervals as in the following corollary.

\begin{corollary}%[Disjoint intervals]
\label{cor:DisjointIntervals}
     Let $(\lambda_n)_{n=1}^\infty$ be the eigenvalues of $T$ in $[m_1, \lambda_{e+})$. Let $\mathcal{I}$ be a finite family of pairwise disjoint open bounded intervals for which $\max_{I\in\mathcal{I}}\Msup_{1,I} \leq \lambda_{e+}$ and let $(\widetilde{\lambda}_n)_{n=1}^\infty$ be the increasing sequence of eigenvalues, counted with multiplicity, of the operator
$$\bigoplus_{I\in\mathcal{I}}f^{-1}_I\left(B_IB^*_I\right).$$
    Then $\lambda_n\leq\widetilde{\lambda}_n$ for $n=1,\, \ldots,\,  \sum_{I\in\mathcal{I}} \dim \mL_{[\Msup_{1,I},\lambda_{e+})}f^{-1}_I(B_I B_I^*)$ and
    $$\dim \mL_{[m_1, \lambda)}T\geq\sum_{I\in\mathcal{I}} \dim \mL_{[\Msup_{1,I},\lambda)}f^{-1}_I(B_I B_I^*)=\sum_{I\in\mathcal{I}} \dim \mL_{[m_1,\lambda)}f^{-1}_I(B_I B_I^*),\qquad \lambda\in\left[m_1, \lambda_{e+}\right].$$
\end{corollary}

Assuming some regularity of $W$, we can estimate the number of eigenvalues of $B_I B_I^*$ as the next remark shows.

\begin{remark}[Number of eigenvalues of $B_I B_I^*$]
   \label{rem:numberEVofBB}
    If $W\in C^1(\overline{I})$, then
    $B_I B_I^* = -\frac{\rd^2}{\rd x^2} + W' + W^2$ with $\mD(B_I B_I^*)=H^2(I)\cap H^1_0(I)$.
    Let $q(x) \coloneqq  W'(x) + W^2(x)$. 
    It follows from the min-max principle that
   \begin{equation*}
      \inf_{x\in I}\, q + \left( \frac{n\pi}{\abs{I}}  \right)^2
      \le \beta_{n,I} \le 
      \sup_{x\in I}\, q + \left( \frac{n\pi}{\abs{I}}  \right)^2,
   \end{equation*}
   where $\abs{I}$ is the length of $I$. 
   Therefore, provided that $f_I(\lambda)>\sup_{x\in I}q$,  we have
   $$\dim \mL_{[0,f_I(\lambda))}B_I B_I^*\geq\max\left\{n\in\N_0 \,:  \, \sup_{x\in I}\ q + \left( \frac{n\pi}{\abs{I}}  \right)^2 < f_I(\lambda)\right\}
   =\ceil{ \frac{\abs{I}}{\pi}\sqrt{f_I(\lambda)-\sup_{x\in I}\, q} }-1.$$
   Hence the longer the interval $I$, the more eigenvalues of $B_I B_I^*$ can ``fit into it'' before $\beta_{n,I} \ge f_I(\lambda_{e+})$.
   Note also that the inequalities above in are in fact equalities if $W$ is constant in $I$.
\end{remark}

Together with Corollary~\ref{cor:DisjointIntervals} this leads to the following lower bound for the number of eigenvalues of $T$.

\begin{corollary}
\label{cor:numberEV}
Under the conditions of Corollary~\ref{cor:DisjointIntervals} and Remark~\ref{rem:numberEVofBB}, the number of eigenvalues of $T$ in $[m_1, \lambda)$ for $\lambda\in [m_1, \lambda_{e+}]$ is
   $$\dim \mL_{[m_1, \lambda)}T
   \geq
   \sum_{I\in\mathcal{I}} \dim \mL_{[0, f_I(\lambda) )} 
   B_I B_I^*
   \ge \sum_{ \substack{ I\in\mathcal{I} \\ f_I(\lambda) > \sup_I q } } 
   \left( \ceil{ \frac{\abs{I}}{\pi}\sqrt{f_I(\lambda)-\sup_{x\in I}\, q}} -1 \right).
   $$
\end{corollary}
%%%}}}

For the next theorem, and subsequent examples, we introduce two families of functions $(M_1( {}\cdot{},\, t))_{t\in J}$ and $(M_2( {}\cdot{},\, t))_{t\in J}$ indexed by some interval $J$. We assume that $M_1({}\cdot{},\, t),\ M_2({}\cdot{},\, t),\ W({}\cdot{})$ satisfy \ref{item:A1}, \ref{item:A2}, \ref{item:A3} for all $t\in J$ and define the operator $T(t)$, and every other related quantity, with $M_1({}\cdot{},  t), M_2({}\cdot{}, t)$ instead of $M_1, M_2$.
\begin{theorem} %%%{{{
   \label{thm:MonotonOfEigenvalues}
   Suppose that $M_1(x,t)$ is decreasing and $M_2(x,t)$ is increasing as a function of $t\in J$. Then, the $n$th eigenvalue of $T(t)$ in $[m_1(t),\lambda_{e+}(t))$ satisfies
\begin{equation*}
      \lambda_n(t_1)\ge \lambda_n(t_2), \qquad t_1\leq t_2,
\end{equation*}
provided both exist.
\end{theorem}

\begin{proof}
   The claim follows immediately from the variational characterisation \eqref{eq:def:mun} of the eigenvalues if we note that for all $\phi\in H_1(\R)$
   \begin{align*}
      \mfs(\lambda,t_1)[\phi]
      & = \int_\R (M_1({}\cdot{},\, t_1) - \lambda) \abs{\phi}^2 \rd x + 
      \int_\R \frac{ \abs{-\phi' + W\phi}^2 }{ M_2({}\cdot{},\, t_1)+\lambda } \rd x\\
      &\ge \int_\R (M_1({}\cdot{},\, t_2) - \lambda) \abs{\phi}^2 \rd x + 
      \int_\R \frac{ \abs{-\phi' + W\phi}^2 }{ M_2({}\cdot{},\, t_2)+\lambda } \rd x
      = \mfs(\lambda,t_2)[\phi].
      \qedhere
   \end{align*}
\end{proof}

\begin{example}
As a special case we consider 
$W(x)\to 0$ for $|x|\to\infty$, $M_2(x,t)=M_2(x)$ and $M_1(x,t) = \Msup_1 - t g(x)$ where $\Msup_1$ is constant and $g\ge 0$ is a bounded measurable function with $\lim_{|x|\to\infty} g(x) = 0$.
Note that in this case $\lambda_{e+}(t) = \Msup_1$, $m_1(t)=\Msup_1-t\norm{g}_\infty$
and therefore $T(0)$ does not have eigenvalues in $[m_1(0), \lambda_{e+}(0))=\emptyset$.
As $t$ increases, the interval $[m_1(t),\, \lambda_{e+}(t)) = \left[\Msup_1-t\norm{g}_\infty,\, \Msup_1\right)$ becomes larger.
Moreover, we assume that $g(x)\ge 1$ for $x$ in some interval $I$, then we obtain for all $\phi\in H^1_0(I)$ that
\begin{align*}
   \mfs(\lambda,t)[\phi]
   &\ge ( \Msup_1 - t\norm{g}_\infty - \lambda) \|\phi\|^2 + \frac{\scalar{B_I^*\phi}{B_I^*\phi}}{\Msup_{2,I} + \lambda }
\end{align*}
and
\begin{align*}
   \mfs(\lambda,t)[\phi]
   & \le (\Msup_1 -t -\lambda) \|\phi\|^2  
   + \frac{\scalar{B_I^*\phi}{B_I^*\phi}}{m_{2,I} + \lambda }.
\end{align*}
In particular for $\lambda = \lambda_{e+}(t) = \Msup_1$ we obtain
\begin{align*}
   \mfs(\Msup_1,t)[\phi]
   & \le - t \|\phi\|^2  
   + \frac{\scalar{B_I^*\phi}{B_I^*\phi}}{m_{2,I} + \Msup_1 } .
\end{align*}
Therefore, by Corollary~\ref{cor:abstractEVexistencis}, as soon as $t > \frac{\beta_{1,I}}{ m_{2,I} + \Msup_1}$,  the operator $T(t)$ has at least one eigenvalue in $[m_1(t),\Msup_1)$.
\end{example}

\begin{example}[Analytic example]
\label{ex:Analytic}
As a toy example we consider the Dirac operator with
\begin{equation*}
    W(x) = 0,\quad
    M_1(x,t) = M - t \chi_{[-1,1]}(x), \quad
    M_2(x,t) = M + \gamma t \chi_{[-1,1]}(x),
\end{equation*}
for some $M>0$ and $\gamma \in \{ 0, 1, -1\}$.
See Figure~\ref{fig:potentialwell}.
If $t=0$, we obtain the free Dirac operator whose spectrum is known to be $\sigma(T(0)) = \essspec(T(0))= (-\infty, -M]\cup [M,\infty)$ and $\discspec(T(0)) = \emptyset$.
So from now on we assume $t> 0$.

\begin{figure}[!hb]
\begin{subfigure}{.33\linewidth}
\centering
\begin{tikzpicture}[scale=.75, transform shape]

   \tikzmath{
   \MM = 1;   %% mass
   \leftend = -2.5;  
   \rightend = 2.5; 
   }

    %% Axis
      \draw[->] (\leftend,0)--(\rightend,0) node [right] {$x$};
      \draw[->] (0,-1.7)--(0,1.5);
   
      \draw[red] (\leftend,\MM)  -- (-1,\MM) -- (-1,{\MM-0.4}) -- (1,{\MM-0.4}) -- (1,\MM) -- (\rightend,\MM) node[right] {$M_1$};
      
      \draw[blue] (\leftend,-\MM) -- (\rightend,-\MM) node[right] {$-M_2$};
\end{tikzpicture}
\caption{$\gamma=0$.}
\end{subfigure}%
\begin{subfigure}{.33\linewidth}
\centering
\begin{tikzpicture}[scale=.75, transform shape]

   \tikzmath{
   \MM = 1;   %% mass
   \leftend = -2.5;  
   \rightend = 2.5; 
   }

    %% Axis
       \draw[->] (\leftend,0)--(\rightend,0) node [right] {$x$};
      \draw[->] (0,-1.7)--(0,1.5);
   
      \draw[red] (\leftend,\MM)  -- (-1,\MM) -- (-1,{\MM-0.4}) -- (1,{\MM-0.4}) -- (1,\MM) -- (\rightend,\MM) node[right] {$M_1$};
      
      \draw[blue] (\leftend,-\MM)  -- (-1,-\MM) -- (-1,{-\MM-0.4}) -- (1,{-\MM-0.4}) -- (1,-\MM) -- (\rightend,-\MM) node[right] {$-M_2$};
\end{tikzpicture}
\caption{$\gamma=1$.}
\end{subfigure}%
\begin{subfigure}{.33\linewidth}
\centering
\begin{tikzpicture}[scale=.75, transform shape]

   \tikzmath{
   \MM = 1;   %% mass
   \leftend = -2.5;  
   \rightend = 2.5; 
   }

    %% Axis
   \draw[->] (\leftend,0)--(\rightend,0) node [right] {$x$};
   \draw[->] (0,-1.7)--(0,1.5);

   \draw[red] 
   (\leftend,\MM)  -- (-1,\MM) -- (-1,{\MM-0.4}) -- (1,{\MM-0.4}) -- (1,\MM) -- (\rightend,\MM) node[right] {$M_1$};
   
   \draw[blue] 
   (\leftend,-\MM)  -- (-1,-\MM) -- (-1,{-\MM+0.4}) -- (1,{-\MM+0.4}) -- (1,-\MM) -- (\rightend,-\MM) node[right] {$-M_2$};
\end{tikzpicture}
\caption{$\gamma=-1$.}
\end{subfigure}%
\caption{Graphs of $M_1$ and $-M_2$ from Example~\ref{ex:Analytic} for 
$\gamma\in\{0,1,-1\}$.}
\label{fig:potentialwell}
\end{figure}
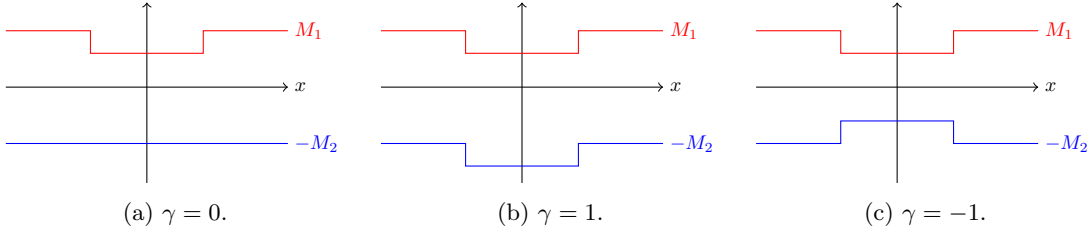
\smallskip

In order to satisfy \ref{item:A2}, we 
have to restrict the parameter $t$ to $t\in (0, 2M)$ if $\gamma \in \{0,1\}$ and $t\in (0, M)$ if $\gamma = -1$.
Then it follows from Proposition~\ref{prop:EssspecResT} that 
\begin{equation*}
   \essspec(T(t)) = (-\infty, -M]\cup [M, \infty),
   \qquad
   \rho(T(t)) \supseteq 
   \begin{cases}
   (-M,\, M-t), \qquad &\text{if }\ \gamma \in \{0, 1\},\\
      (-M+t,\, M-t), \qquad &\text{if }\ \gamma = -1.
   \end{cases}
\end{equation*}
Consequently, if there are eigenvalues in the essential spectral gap, they must lie in
\begin{equation*}
    \Sigma_{\gamma, t} =
   \begin{cases}
   [M-t,\, M), \qquad &\text{if }\ \gamma \in \{0, 1\},\\
      (-M,\, -M+t] \cup [M-t,\, M), \qquad &\text{if }\ \gamma = -1.
   \end{cases}
\end{equation*}

In order to solve the eigenvalue equation analytically, we 
let $\lambda\in\Sigma_{\gamma, t}$ and
rewrite $(T(t) - \lambda)\Psi = 0$ as the system
\begin{equation*}
    (M_1 - \lambda) \psi_1 + \psi_2' = 0,
    \qquad
    -\psi_1' + (-M_2 - \lambda) \psi_2 = 0.
\end{equation*}
Differentiating the second equation and inserting into the first, we obtain the equation
\begin{equation*}
    \psi_1''  =  (M_1 - \lambda)(M_2 + \lambda) \psi_1.
\end{equation*}

Note that $\nuext(\lambda) \coloneqq \sqrt{M^2-\lambda^2}>0$ and $\nuint(\lambda,t,\gamma) \coloneqq \sqrt{(\lambda-M+t)(\lambda+M+\gamma t)}\geq0$ since $\lambda\in \Sigma_{\gamma,t}$. With these parameters 
the differential equation above becomes
\begin{equation*}
    \psi_1''  =  
    \begin{cases}
       \nuext^2(\lambda) \psi_1,\quad &\text{if}\ |x| > 1,\\
       - \nuint^2(\lambda, t, \gamma) \psi_1,\quad &\text{if}\ |x| < 1.
    \end{cases}
\end{equation*}

Moreover, since $W(\cdot),M_j(\cdot,t)$ are even functions, we know that the components of $\Psi$ have definite parity. This together with $\psi_1\in L_2(\R)$ 
leads us to 
\begin{equation*}
    \psi_1(x) = 
    \begin{cases}
        A e^{-\nuext(\lambda) |x|},\quad & \text{if}\ |x| > 1,\\
        B \cos(\nuint(\lambda, t, \gamma) x),\quad & \text{if}\ |x| < 1,
    \end{cases}
    \quad\text{or}\quad
    \psi_1(x) = 
    \begin{cases}
        \sgn(x) A e^{-\nuext(\lambda) |x|},\quad & \text{if}\ |x| > 1,\\
        B \sin(\nuint(\lambda, t, \gamma) x),\quad & \text{if}\ |x| < 1.
    \end{cases}
\end{equation*}
The continuity of $\psi_1$ and $\psi_1'$ at $x=\pm1$ implies that any eigenvalue $\lambda$ must satisfy $\nuint(\lambda, t, \gamma)\neq0$ and $\tan \left(\nuint(\lambda, t, \gamma)\right) = p\left(\frac{\nuext(\lambda)}{\nuint(\lambda, t, \gamma)}\right)^{p}$ for  $p=\pm 1$, hence
\begin{equation}
\label{eq:specToyEx}
\discspec(T(t))
= \left\{\lambda\in \Sigma_{\gamma, t}\setminus\{\pm(M-t)\}:
\tan \left(\nuint(\lambda,t,\gamma)\right)
= p \left( \frac{\nuext(\lambda)}{ \nuint(\lambda,t,\gamma)} \right)^p,
\quad p = \pm 1
\right\}
\end{equation}
where $p = 1$ corresponds to even solutions and $p = -1$ to odd solutions.
\smallskip

From the equation above we see that a new eigenvalue enters the spectral gap from $\lambda= M$ if $\tan\left(\nuint(\lambda,t,\gamma)\right)$ is $0$ or $\pm\infty$, that is, when
\begin{equation}
\label{eq:ToyExEValueCondition}
   t(2M+\gamma t) \in \left\{ \left( \frac{n \pi}{2} \right)^2 : n\in \N_0 \right\}.
\end{equation}

Now let us consider the different cases for $\gamma$.

\begin{enumerate}[label={\upshape (\roman*)}]

    \item $\gamma = 0$.
    All eigenvalues of $T(t)$ in its essential spectral gap must lie in 
    $\Sigma_{0,t}\setminus\{M-t\} = (M-t,\, M)$.
    By Theorem~\ref{thm:MonotonOfEigenvalues} all eigenvalues are decreasing in $t$, so no eigenvalues can enter the interval from its lower bound $M-t$.
    We see from \eqref{eq:ToyExEValueCondition} that a new eigenvalue emerges from $\lambda = M$ whenever 
    \begin{equation}
      \label{eq:ToyExExactEV0}
       t= t_n = \frac{n^2\pi^2}{8M}.
    \end{equation}
    Therefore, $T(t)$ has exactly $n$ eigenvalues if $t\in (t_{n-1}, t_n)$ for $n\in \N$ as long as $t_n \le 2M$.
    Since $t< 2M$, this shows that the maximal number of eigenvalues in $(M-t, M)$ is $\lfloor \frac{4M}{\pi} \rfloor + 1$.
    
    \item $\gamma = 1$.
    All eigenvalues of $T(t)$ in its essential spectral gap must lie in  $\Sigma_{1,t}\setminus\{M-t\} = (M-t,\, M)$.
    Again, by Theorem~\ref{thm:MonotonOfEigenvalues} all eigenvalues are decreasing in $t$, so no eigenvalues can enter the interval from its lower bound $M-t$.
    We see from \eqref{eq:ToyExEValueCondition}, taking into account the condition $t>0$, that a new eigenvalue emerges form $\lambda=M$ whenever
    \begin{equation}
      \label{eq:ToyExExactEV1}
       t = t_n = -M +\sqrt{ M^2 + \frac{n^2\pi^2}{4}}.
    \end{equation}
    Therefore, $T(t)$ has exactly $n$ eigenvalues if $t\in (t_{n-1}, t_n)$ for $n\in \N$ as long as $t_n \le 2M$.
    Since $t< 2M$, the maximal number of eigenvalues in $(M-t, M)$ is 
    $\lfloor \frac{4\sqrt{2} M}{\pi} \rfloor + 1$.
    
    \item $\gamma = -1$.
    All eigenvalues of $T(t)$ in its essential spectral gap must lie in  $\Sigma_{-1,t}\setminus\{\pm(M-t)\} = (-M,\, -M+t) \cup (M-t,\, M)$. By \eqref{eq:Ttransformed} we have that $\sigma(T(t))$ is symmetric with respect to $0$, so it is enough to consider eigenvalues in $(M-t,\, M)$. This time we cannot use Theorem~\ref{thm:MonotonOfEigenvalues}, however, implicit differentiation of \eqref{eq:specToyEx} shows
    \begin{multline*}
       \lambda'(t)
       = -\frac{ (M-t)\nuext(\lambda(t))\left[\nuext(\lambda(t))+\nuint^2(\lambda(t), t, -1)\sec^2\left(\nuint(\lambda(t), t, -1)\right)\right] }{ \lambda(t)\left[t(2M-t)+\nuext(\lambda(t))\nuint^2(\lambda(t), t, -1)\sec^2\left(\nuint(\lambda(t), t, -1)\right)\right] }
       \\
       \times
       \left(\frac{\nuext(\lambda(t))}{\nuint(\lambda(t), t, -1)}\right)^{-1+p}.
    \end{multline*}
    This equality, together with $t<M$, implies that the eigenvalues in $(M-t,M)$ are decreasing, and therefore none can enter from $M-t$. Once again, we see from \eqref{eq:ToyExEValueCondition} that a new eigenvalue emerges from $\lambda = M$ whenever
    \begin{equation}
      \label{eq:ToyExExactEV-1}
       t = t_n = M - \sqrt{ M^2 - \frac{n^2\pi^2}{4}}.
    \end{equation}
    
    Therefore, $T(t)$ has $n$ eigenvalues in $(M-t, M)$ if $t\in (t_{n-1}, t_n)$ for $n\in \N$ as long as the right hand side in the formula above is real. This shows that the maximal number of eigenvalues in $(M-t, M)$ is 
    $\lfloor \frac{2 M}{\pi} \rfloor + 1$.    
\end{enumerate}

\begin{remark}
Without the explicit calculations of Example~\ref{ex:Analytic} above, we can at least say the following, using only
Theorem~\ref{thm:teoremafavorito} and Corollary~\ref{cor:numberEV}.

Let $I = [-1,\, 1]$ and let $B_I B_I^* = - \frac{\rd^2}{\rd x^2}$ with Dirichlet boundary conditions.
Its eigenvalues are $\beta_{n,I} = \left( \frac{n\pi}{2} \right)^2$ for $n\in \N$. Note that $\Msup_{1,I}(t) = M-t$ is independent of the value of $\gamma$ and that $m_{2,I}(t) = M+\gamma t$.
   
   With the estimate \eqref{eq:upperboundInterval} from Theorem~\ref{thm:teoremafavorito} we obtain that, for $n\in \N$,
   \begin{equation}
       \label{eq:ToyExUpperbound}
       \lambda_n(t)
       \le 
       \Lambda_{\gamma, n}(t)
       \coloneqq \frac{1}{2} \left( -t(1+\gamma) + \sqrt{ \left[2M-t(1-\gamma)\right]^2 + (n\pi)^2 } \right).
   \end{equation}
   Since $\Lambda_{\gamma,n}(t)$ is decreasing in $t$, we can guarantee that $T(t)$ has at least $n$ eigenvalues in $[M-t,M)$ for all $t$ such that $\Lambda_{\gamma, n}(t)<M$. More precisely:

\begin{enumerate}[label={\upshape (\roman*)}]
   \item $\gamma =0$.
   As $t$ increases from $0$ to $2M$, the $n$th eigenvalue must have entered $[M-t, M)$ before
   \begin{equation*}
       t = \widetilde t_{n} = \frac{n^2\pi^2}{8M}.
   \end{equation*}
   Therefore, $T(t)$ has at least $n$ eigenvalues in $[M-t,\, M)$ if $t\in (\widetilde t_n,\, \widetilde t_{n+1})$ as long as $\widetilde t_n < 2M$.
   % In particular, the maximal number of eigenvalues that we can guarantee for $t$ large enough, is $\lfloor \frac{4M}{\pi} \rfloor$.
      
   \item $\gamma = 1$.
   As $t$ increases from $0$ to $2M$, the $n$th eigenvalue must have entered $[M-t, M)$ before
   \begin{equation*}
       t = \widetilde t_n 
       = -M + \sqrt{ M^2 + \left( \frac{n\pi}{2} \right)^2 } \, .
   \end{equation*}
   Therefore, $T(t)$ has at least $n$ eigenvalues in $[M-t,\, M)$ if $t\in (\widetilde t_n,\, \widetilde t_{n+1})$ as long as $\widetilde t_n < 2M$.
   % In particular, the maximal number of eigenvalues that we can guarantee for $t$ large enough, is $\lfloor \frac{4\sqrt{2}M}{\pi} \rfloor$.
   
   \item $\gamma = -1$.
   As $t$ increases from $0$ to $M$, the $n$th eigenvalue must have entered $[M-t, M)$ before
   \begin{equation*}
       t = \widetilde t_n 
       = M - \sqrt{ M^2 - \left( \frac{n\pi}{2} \right)^2 }\, .
   \end{equation*}
   Therefore, $T(t)$ has at least $n$ eigenvalues in $[M-t,\, M)$ if $t\in (\widetilde t_n,\, \widetilde t_{n+1})$ as long as $\widetilde t_n < M$.
   % In particular, the maximal number of eigenvalues that we can guarantee for $t$ large enough, is $\lfloor \frac{2M}{\pi} \rfloor$.

\end{enumerate}
Equivalently, we can use Corollary~\ref{cor:numberEV} to estimate the number of eigenvalues of $T(t)$ in $[M-t, M)$ by
\begin{equation*}
  \dim\mL_{[M-t, M)}T(t) 
  \ge  \ceil{\frac{2}{\pi} \sqrt{ t\left(2M+\gamma t\right) }} -1 .
\end{equation*}
From here we find a lower bound for the maximum possible number of eigenvalues as
\begin{equation*}
     \lim_{t\to(2M)^-}\dim\mL_{[M-t, M)}T(t) 
     \ge 
     \begin{cases}
        \floor{\frac{4M}{\pi} },\quad
        &\text{if } \gamma = 0, \\[2ex]
        \floor{\frac{4\sqrt{2} M}{\pi} },\quad
        &\text{if } \gamma = 1,
     \end{cases}
\end{equation*}
and
\begin{equation*}
     \lim_{t\to M^-}\dim\mL_{[M-t, M)}T(t) 
     \ge \floor{\frac{2M}{\pi} },
     \quad \text{if } \gamma = -1.
\end{equation*}
Note that in all three cases $\widetilde t_{n} = t_{n}$ from 
\eqref{eq:ToyExExactEV0}, \eqref{eq:ToyExExactEV1}, \eqref{eq:ToyExExactEV-1}
and that, for any given $t$, the guaranteed number of eigenvalues provided by Theorem~\ref{thm:teoremafavorito} and Corollary~\ref{cor:numberEV} is one less than the actual number of eigenvalues.
\end{remark}

Figure~\ref{fig:analyticexample} shows the eigenvalue trajectories 
and their corresponding upper bounds given in \eqref{eq:specToyEx} and \eqref{eq:ToyExUpperbound} respectively, for $M=4$ and $\gamma\in\{0,1,-1\}$.
\end{example}

\begin{figure}[!hb]
\begin{subfigure}{.5\linewidth}
\centering
\includegraphics[width=\linewidth]{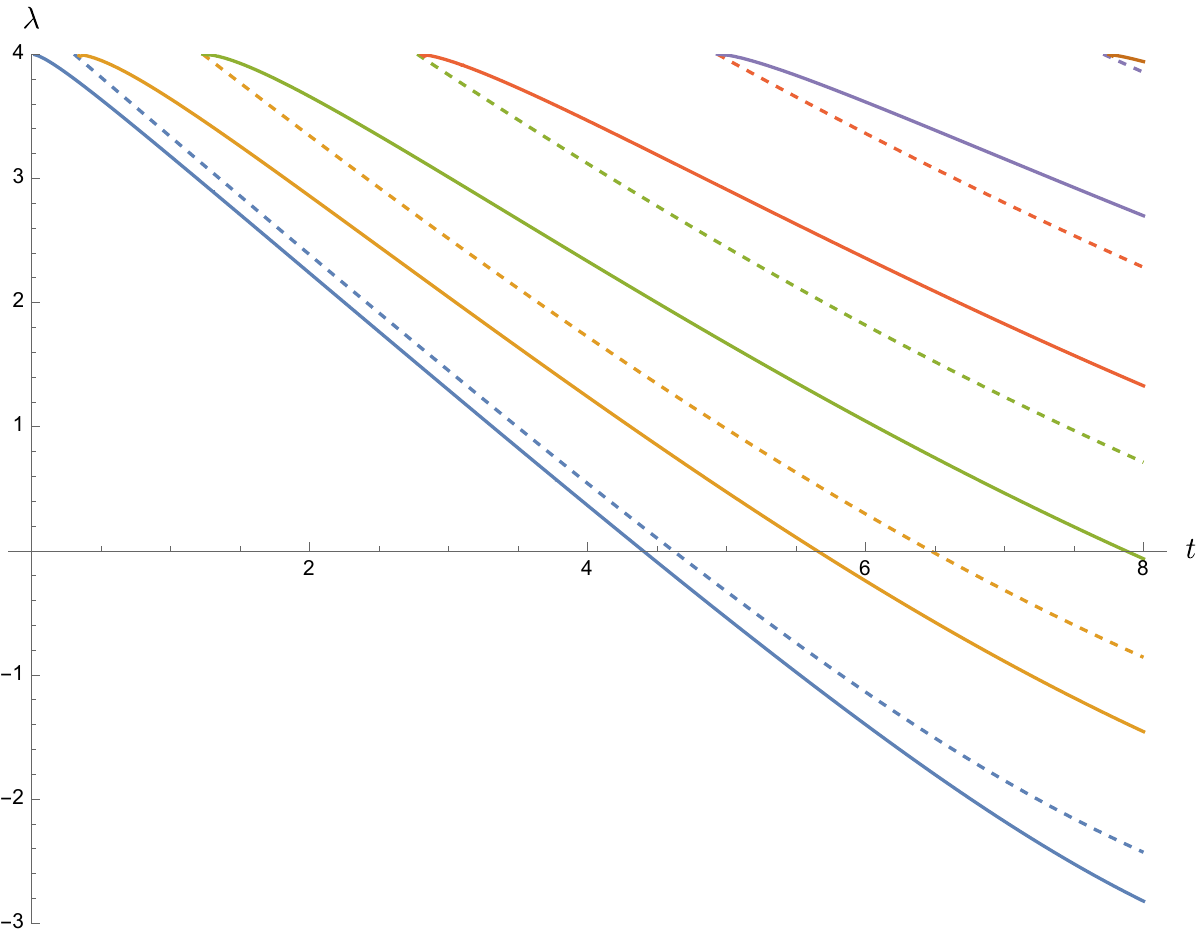}
\caption{$M=4$, $\gamma=0$.}
\end{subfigure}%
\begin{subfigure}{.5\linewidth}
\centering
\includegraphics[width=\linewidth]{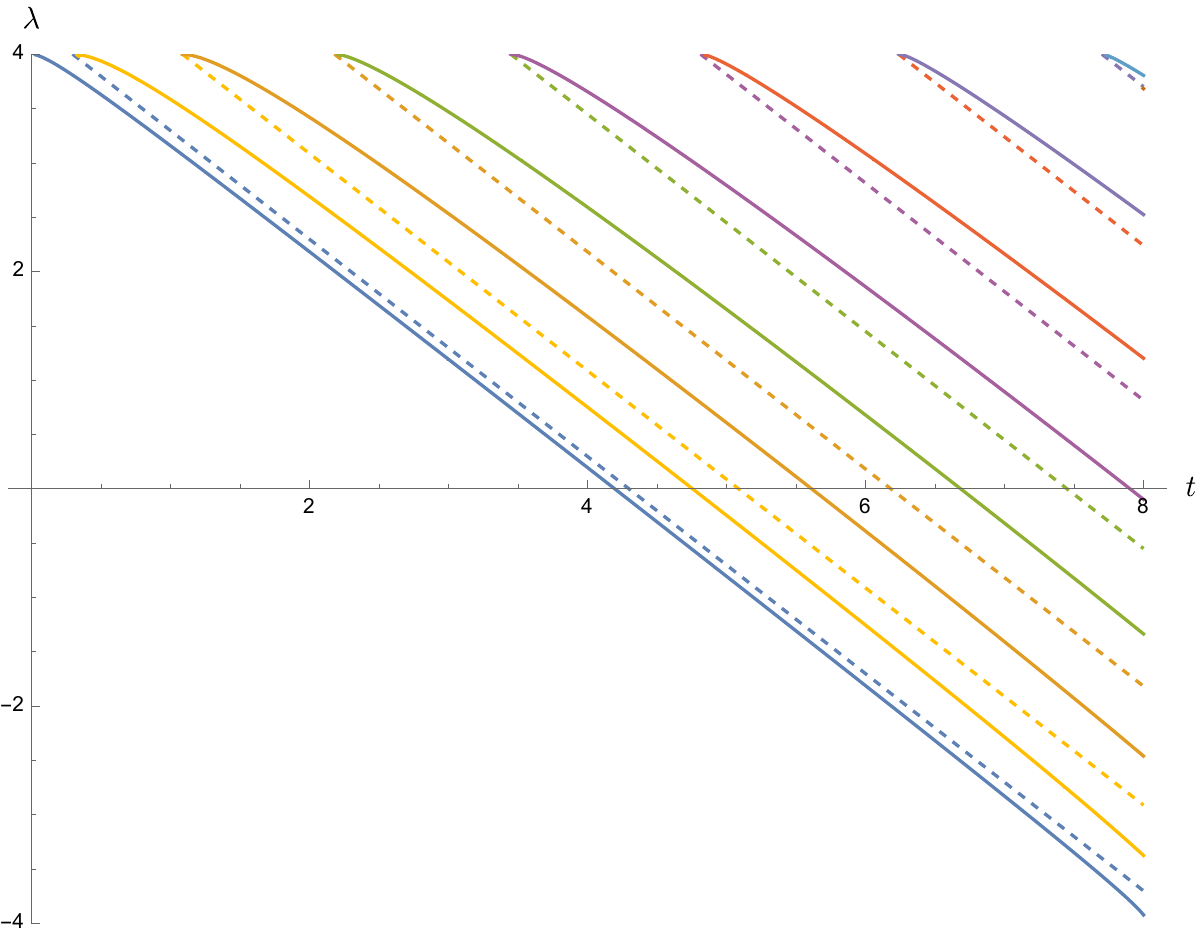}
\caption{$M=4$, $\gamma=1$.}
\end{subfigure}\\[1ex]
\begin{subfigure}{\linewidth}
\centering
\includegraphics[width=0.5\linewidth]{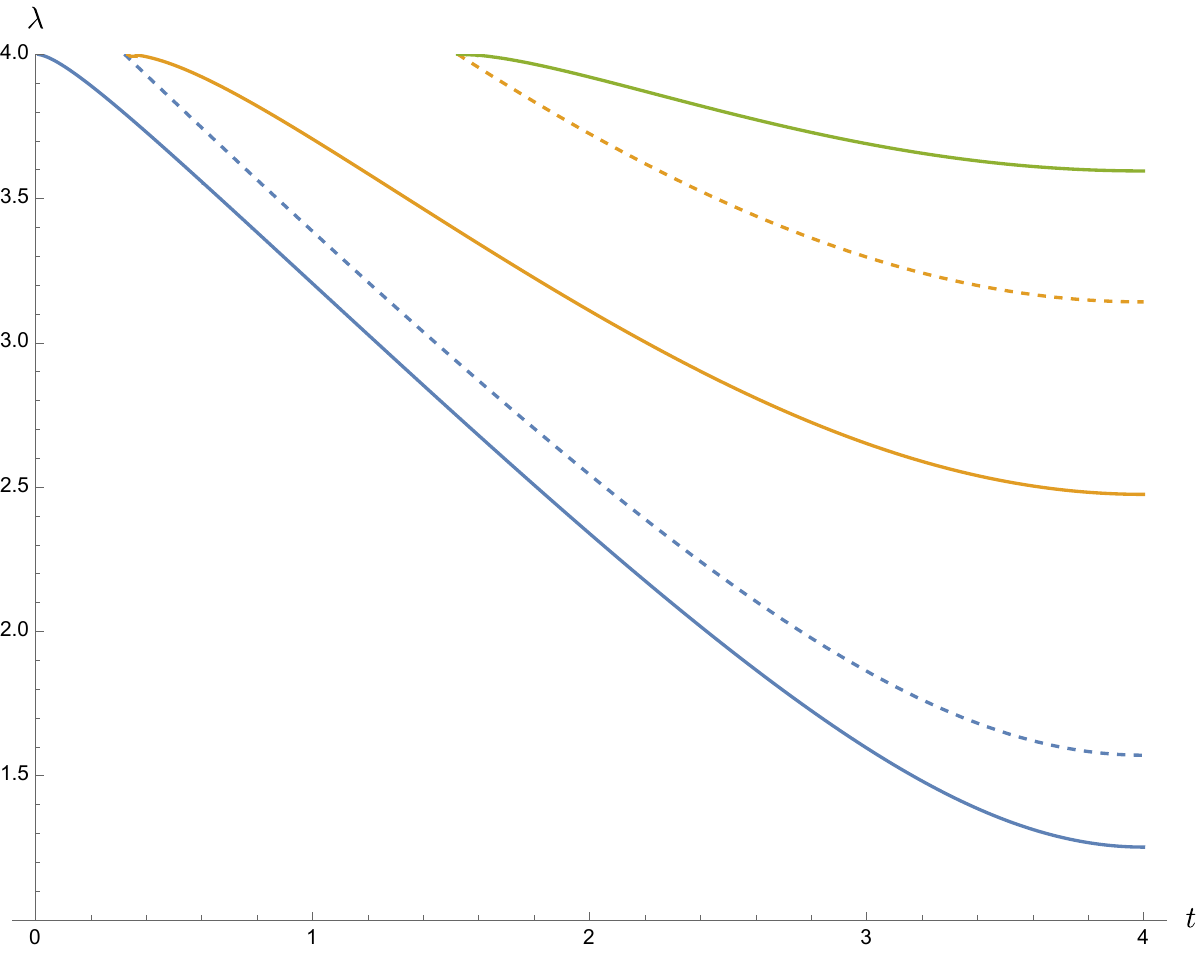}
\caption{$M=4$, $\gamma=-1$.}
\end{subfigure}
\caption{Eigenvalues (continuous) and their corresponding upper bounds (dashed) from Example~\ref{ex:Analytic} for $M=4$ and $\gamma\in\{0,1,-1\}$.}
\label{fig:analyticexample}
\end{figure}

\clearpage

\begin{example}[Numerical example: one-dimensional hydrogenic Dirac operator]
\label{ex:Shadi}
We apply our results to the one-dimensional hydrogen-type Dirac operator studied in 
\cite{DasguptaKhuranaTahvildarZadeh2023}. 
The functions $M_j$ and $W$ are
\begin{equation*}
   W(x)=0,\quad 
   M_1(x,t)=M-\frac{t}{2}e^{-\abs{x}},\quad M_2(x,t)=M+\frac{t}{2}e^{-\abs{x}},
\end{equation*}
for some $M>0$ and $t\in(0,4M)$. 
To the best of our knowledge, the eigenvalues cannot be computed analytically.
In \cite{DasguptaKhuranaTahvildarZadeh2023}, the authors compute the eigenvalues numerically through an associated dynamical system. 
\smallskip

Here we use the variational principle presented in Theorem~\ref{thm:varprinciple:abstract} to calculate them numerically.
Note that in this example we have $\lambda_{e+}(t)=m_2(t)=M$ and $m_1(t)=M-\frac{t}{2}$. Theorem~\ref{thm:IntegrabilityCond} guarantees the existence of at least one eigenvalue in $[m_1(t),\lambda_{e+}(t))$ and Theorem~\ref{thm:MonotonOfEigenvalues} tells us that those eigenvalues are decreasing in $t$. In Figure~\ref{fig:Shadi} we show the eigenvalue trajectories for $M=1$. 
Our plot reproduces the upper-left corner of \cite[Fig. 5]{DasguptaKhuranaTahvildarZadeh2023}, however we cannot go further than $t=4$ because of the condition \ref{item:A2}. 
\end{example}

\begin{figure}[!hb]
\begin{subfigure}{.35\linewidth}
\centering
\includegraphics[width=\linewidth]{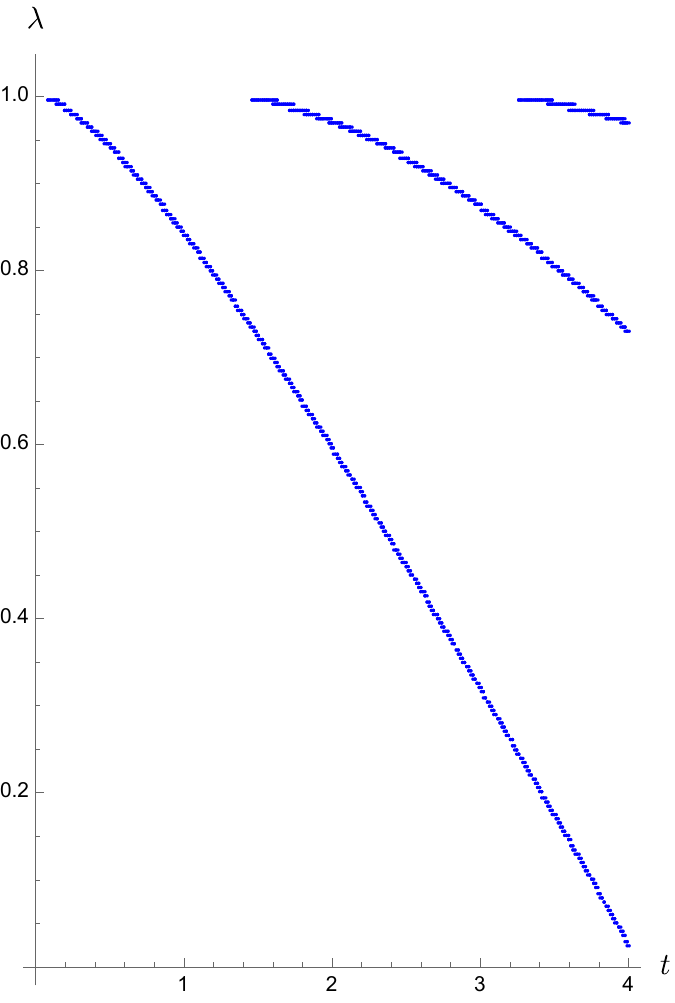}
\caption{Eigenvalues computed using the variational principle from Theorem~\ref{thm:varprinciple:abstract}.}
\end{subfigure}%
\begin{subfigure}{.65\linewidth}
\centering
\includegraphics[width=\linewidth]{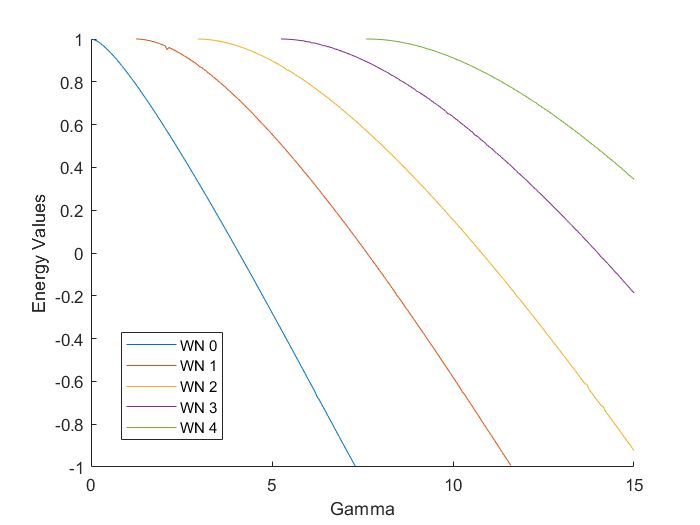}
\caption{Figure 5 of \cite{DasguptaKhuranaTahvildarZadeh2023}.}
\end{subfigure}%
\caption{Numerically computed eigenvalue trajectories from Example~\ref{ex:Shadi} for $M=1$. The parameter $t$ in (a) corresponds to Gamma in (b).}
\label{fig:Shadi}
\end{figure}

%%%}}}

%%%}}}

\section{The Dirac operator on the halfline}
\label{sec:DiracHalfLine}%%%{{{

Let us briefly discuss the Dirac operator on the halfline $\R_+$.
Recall from Proposition~\ref{prop:defT} that we need to specify a boundary condition at $x=0$ in order to obtain a selfadjoint realization of $\tau$ on the halfline. To simplify notation, we write $T_\alpha$ instead of $T_{\alpha+}$ since in this section we only work on $\R_+$.
Recall that the boundary value parameter is $\alpha\in [0,\pi)$.
Moreover, in this section we set $m_j \coloneqq \inf_{x\in\R_+} M_j(x)$.
\smallskip

For $\alpha\in\{0,\frac{\pi}{2}\}$, the domain of $T_\alpha$ is a direct sum given by
\begin{align*}
   \mD(T_{0}) 
   & = \left\{ 
   \Psi \in H^1(\R_{+})^2 \ :\
   \Psi(0) = c
   \begin{pmatrix}
      1 \\ 0
   \end{pmatrix}\text{ for some } c\in\C
   \right\}
   =
   H^1(\R_{+}) \oplus H_0^1(\R_{+}),\\
   \mD(T_{\frac{\pi}{2}}) 
   & = \left\{ 
   \Psi \in H^1(\R_{+})^2 \ :\
   \Psi(0) = c
   \begin{pmatrix}
      0 \\ 1
   \end{pmatrix}\text{ for some } c\in\C
   \right\}
   =
   H_0^1(\R_{+}) \oplus H^1(\R_{+}),
\end{align*}

and $T_\alpha$ can be written as the block matrix
\begin{equation*}
   T_\alpha =
   \begin{pmatrix}
       M_1 & B_\alpha \\  B_\alpha^* & -M_2
   \end{pmatrix},
\end{equation*}
where 
\begin{equation*}
    B_\alpha = \frac{\rd}{\rd x} + W,
    \qquad
    B_\alpha^* = -\frac{\rd}{\rd x} + W,
\end{equation*}
have domains
\begin{equation*}
   \mD(B_0) = \mD(B_{\frac{\pi}{2}}^*) = H^1_0(\R_+),
   \qquad
   \mD(B_{0}^*) = \mD(B_{\frac{\pi}{2}})  = H^1(\R_+).
\end{equation*}

For all other cases of $\alpha$ the domain of $T_\alpha$ cannot be written as the direct sum of subspaces of $L_2(\R_+)\oplus\{0\}$ and $\{0\}\oplus L_2(\R_+)$ and therefore the variational principle presented in Section~\ref{sec:varprinciple} cannot be applied.

\begin{remark}
    Even if the variational principle does not apply for $\alpha \notin \{0, \frac{\pi}{2}\}$, we can use it to obtain bounds for the eigenvalues. 
    It suffices to note that the eigenvalues, if they exist, are decreasing as a function of $\alpha\in[0,\pi)$ with
    $\frac{\rd \lambda}{\rd\alpha} = -\|\Psi_{\lambda,\alpha}\|^{-2}$
    where $\Psi_{\lambda,\alpha}$ is the eigenfunction for the eigenvalue $\lambda$ and boundary condition
    $\Psi_{\lambda, \alpha}(0) = \begin{pmatrix} \cos\alpha \\ \sin\alpha \end{pmatrix}$,
    see \cite[Sec. 15.6]{weidmannII}.
\end{remark}

Let us consider the case $\alpha=0$.
The case $\alpha=\frac{\pi}{2}$ can be dealt with similarly.
\smallskip

According to Proposition~\ref{prop:essspecT}, the  essential spectrum of $T_0$ is 
$\essspec(T_0) = (-\infty,\, \lambda_+^-]\cup [\lambda_+^+,\, \infty)$. Clearly, $T_0$ satisfies all the assumptions from Section~\ref{sec:varprinciple}, 
in particular, the form $\mfs_0(\lambda)$ is given by 
\begin{equation*}
    \mfs_0(\lambda)[\phi] = 
    \scalar{ (M_1-\lambda)\phi }{\phi}
    +  \scalar{ (M_2+\lambda)^{-1} B_0^* \phi }{ B_0^* \phi}
\end{equation*}
for all $\re \lambda > -m_2$ and $\phi\in\mD(\mfs_0) = H^1(\R_+)$.  Moreover, from Corollary~\ref{cor:resolventab} we obtain that $(-m_2,\, m_1)\subseteq \rho(T_0)$.
\smallskip

For the eigenvalues of $T_0$ in $[m_1,\, \lambda_+^+)$, 
Theorems~\ref{thm:teoremafavoritoglobal} and \ref{thm:teoremafavorito} and their corollaries hold with the obvious modifications. They can also be used for eigenvalues in $(\lambda_+^-,-m_2]$ together with \eqref{eq:Ttransformed} or \eqref{eq:SchurD}, however some care is needed with the domains since $\mD(B_0) \neq \mD(B_0^*)$. 
\smallskip

Theorem~\ref{thm:IntegrabilityCond} on the existence of an eigenvalue in $[m_1,\,\lambda_+^+)$ holds if integration over $\R$ is replaced by integration over $\R_+$. For the proof, the same functions $\phi_n$ restricted to $\R_+$ can be used. Also Theorem~\ref{thm:aEqualsEssspec} can be adapted. Note however, that the proofs are not valid for eigenvalues in $(\lambda_+^-,\, -m_2]$ because the restrictions of the functions $\phi_n$ to $\R_+$ do not belong to $\mD(B_0)$. 
\smallskip

We finish this section with an example that can be solved analytically.
\begin{example}[Analytic example]
\label{ex:AnalyticHalfline}
We consider the Example~\ref{ex:Analytic} but restricted to the halfline $\R_+$.

Let
\begin{equation*}
    W(x) = 0,\quad
    M_1(x,t) = M - t \chi_{[0,1]}(x), \quad
    M_2(x,t) = M + \gamma t \chi_{[0,1]}(x)
\end{equation*}
for some $M>0$ and $\gamma \in \{ 0, 1, -1\}$.
\smallskip

In order to satisfy \ref{item:A2}, we 
have to restrict the parameter $t$ to
$t\in [0, 2M)$ if $\gamma \in \{0,1\}$
and $t\in [0, M)$ if $\gamma = -1$. Then it follows from Proposition~\ref{prop:essspecT} that
\begin{equation*}
   \essspec(T_\alpha(t)) = (-\infty, -M]\cup [M, \infty),\qquad\alpha\in[0,\pi).
\end{equation*}
In order to solve the eigenvalue equation analytically, we 
rewrite $(T_\alpha(t) - \lambda)\Psi = 0$ as the system
\begin{equation*}
    (M_1 - \lambda) \psi_1 + \psi_2' = 0,
    \qquad
    -\psi_1' + (-M_2 - \lambda) \psi_2 = 0.
\end{equation*}
As before, we obtain for $\psi_1$ the differential equation $\psi_1''  =  (M_1 - \lambda)(M_2 + \lambda) \psi_1$, which we write as
\begin{equation*}
    \psi_1''  =  
    \begin{cases}
       \nuext^2(\lambda) \psi_1,\quad &\text{if}\ x > 1,\\
       - \nuint^2(\lambda, t, \gamma) \psi_1,\quad &\text{if}\ 0 < x < 1,
    \end{cases}
\end{equation*}
with
$\nuext(\lambda) = \sqrt{M^2-\lambda^2}$ and $\nuint(\lambda,t,\gamma) = \sqrt{(\lambda-M+t)(\lambda+M+\gamma t)}$. 
This time, however, we also have the boundary condition at $x=0$
\begin{equation*}
   \psi_1'(0) =  -(M+\gamma t + \lambda)\tan(\alpha) \psi_1(0).
\end{equation*}
We note that $\psi_1\in L_2(\R_+)$ immediately implies that any eigenvalue $\lambda$ of $T_\alpha(t)$ satisfies $\nuext(\lambda)>0$, or equivalently $\abs{\lambda}<M$.
\smallskip

\begin{remark}\label{rem:EigenvalueAlpha}
   The operator $T_\alpha(0)$ has no eigenvalue if $\alpha = 0$ or $\alpha \in [\frac{\pi}{2}, \pi)$.
   It has exactly one eigenvalue $\lambda(\alpha)$ for every $\alpha\in (0,\, \frac{\pi}{2})$; it is $\lambda(\alpha)=M\cos(2\alpha)$.
\end{remark}
\begin{proof}
If $t=0$, then the eigenvalue equation for $\psi_1$ is $\psi_1'' = (M^2-\lambda^2)\psi_1$ whose only square integrable solution is $\psi_1(x) = A e^{-\nuext(\lambda) x}$.
Now the boundary condition at $x=0$ reads
$$\nuext(\lambda) = (M+\lambda)\tan\alpha.$$
Since $\nuext(\lambda)$ and $M+\lambda$ are positive and finite, the previous equation has a solution only if $\alpha\in (0, \frac{\pi}{2})$. Solving for $\lambda$ gives the claim.
\end{proof}

From now on let us assume $t>0$. Solving for $\psi_1$ leads us to
\begin{equation*}
    \psi_1(x) = 
    \begin{cases}
        A e^{-\nuext(\lambda) x},\quad & \text{if}\ x > 1,\\
        B \sin(\nuint(\lambda, t, \gamma) x) 
        + C \cos(\nuint(\lambda, t, \gamma) x),\quad & \text{if}\ 0 < x < 1.
    \end{cases}
\end{equation*}
Continuity of $\psi_1,\, \psi_1' $ at $x=1$ and the boundary condition at $x=0$ give
\begin{align*}
   A e^{-\nuext(\lambda)}&= B \sin(\nuint(\lambda, t, \gamma) )
    + C \cos(\nuint(\lambda, t, \gamma)),
    \\
      -A \nuext(\lambda)e ^{-\nuext(\lambda)}&=\nuint(\lambda, t, \gamma)\left[B\cos(\nuint(\lambda, t, \gamma) )
    - C \sin(\nuint(\lambda, t, \gamma))\right],
    \\
    B \nuint(\lambda, t, \gamma) &= -C (M+\gamma t + \lambda) \tan\alpha,
\end{align*}
from which follows that $\lambda\in (-M,M)$ is an eigenvalue of $T_\alpha(t)$ if and only if $\nuint(\lambda, t, \gamma)>0$ and
\begin{equation}
\label{eq:BChalfline}
    \nuext(\lambda) 
     = 
    \nuint(\lambda, t, \gamma)
    \frac{ \tan(\nuint(\lambda, t, \gamma) )
    + (M+\gamma t + \lambda) \nuint(\lambda, t, \gamma)^{-1} \tan\alpha   }{
    1
    - \tan(\nuint(\lambda, t, \gamma))(M+\gamma t+\lambda) \nuint(\lambda, t, \gamma)^{-1} \tan\alpha}.
\end{equation}
We remark that for $\alpha = 0$ we get from \eqref{eq:BChalfline} the eigenvalues of \eqref{eq:specToyEx} with $p=1$,
and for $\alpha=\frac{\pi}{2}$ we get the eigenvalues of \eqref{eq:specToyEx} with $p=-1$.
\smallskip

To see when a new eigenvalue enters the essential spectral gap from $M$ we set $\lambda=M$ in \eqref{eq:BChalfline}.
Then the left hand side vanishes, so the numerator of the right hand side must also vanish
(by multiplying both numerator and denominator by $\cos(\nuint(\lambda,t,\gamma))\cos\alpha$ we can easily discard the possibility of the denominator being $\pm\infty$).
This means that a new eigenvalue enters the spectral gap from $M$ whenever $t$ satisfies (see Figure \ref{fig:neweigenvalue})
\begin{equation}
\label{eq:neweigenvalue}
  \tan \left(\sqrt{ t(2M+\gamma t)}\right) = 
  - \sqrt{\frac{ 2M + \gamma t }{ t }}\, \tan\alpha.
\end{equation}

\begin{figure}[!hb]
\begin{subfigure}{.48\linewidth} %% gamma=0
%%%% Gamma = 0, 0 < t < 2M
\begin{tikzpicture}[scale=.8, transform shape]
\begin{axis}[ymin=-20, ymax=35, xmin=-0.15, xmax=8.5,
axis x line=center,
axis y line=center,
xlabel={$t$},
xlabel style={right},
xtick={4,8},
ytick=\empty,
legend style={%at={(0.5,0.5)},
   anchor=east,legend columns=3},
legend style={font=\tiny},
declare function={ 
rhs(\t,\MM,\GG,\AA) = -sqrt( 2*\MM/\t + \GG)*tan(\AA r);
lhs(\t,\MM,\GG) = tan( sqrt( \t*(2*\MM + \t*\GG) ) r );
}
]
\pgfplotsset{
    % this modifies the `every tick label' style
    tick label style={ font=\footnotesize},
}

\addplot[blue!80!red, domain=0:8, samples=300]{rhs(x, 4,0,0)};
% \addlegendentry{\footnotesize $\alpha=0$}
\addplot[blue!60!red, domain=0:8, samples=300]{rhs(x, 4,0, pi/4)};
% \addlegendentry{\footnotesize $\alpha=\frac{\pi}{4}$}
\addplot[blue!20!red, domain=0:8, samples=300]{rhs(x, 4,0,3/8*pi)};
% \addlegendentry{\footnotesize $\alpha=\frac{3\pi}{8}$}
\addplot[blue!10!red, domain=0:8, samples=300]{rhs(x, 4,0, 5/12*pi)};
% \addlegendentry{\footnotesize $\alpha=\frac{5\pi}{12}$}
\addplot[red!90!green, domain=0:8, samples=300]{rhs(x, 4,0, pi/2+1/16*pi)};
% \addlegendentry{\footnotesize $\alpha=\frac{9\pi}{16}$}
\addplot[red!80!green, domain=0:8, samples=300]{rhs(x, 4,0, pi/2+1/8*pi)};
% \addlegendentry{\footnotesize $\alpha=\frac{5\pi}{8}$}
\addplot[red!40!green, domain=0:8, samples=300]{rhs(x, 4,0, pi/2+3/8*pi)};
% \addlegendentry{\footnotesize $\alpha=\frac{7\pi}{8}$}

\addplot[black, domain=0:1, samples=300, restrict y to domain=-50:50] {lhs(x, 4,0)};
\addplot[black, domain=1:3, samples=300, restrict y to domain=-50:50] {lhs(x, 4,0)};
\addplot[black, domain=3:8, samples=300, restrict y to domain=-50:50] {lhs(x, 4,0)};
\end{axis}

\end{tikzpicture}

\caption{$\gamma=0$.}
\end{subfigure}
\begin{subfigure}{.48\linewidth}  %% gamma = 1
%%%% Gamma = 1, 0 < t < 2M
\begin{tikzpicture}[scale=.8, transform shape]
\begin{axis}[ymin=-20, ymax=35, xmin=-0.15, xmax=8.5,
axis x line=center,
axis y line=center,
xlabel={$t$},
xlabel style={right},
xtick={4,8},
ytick=\empty,
legend style={%at={(0.5,0.5)},
   anchor=east,legend columns=3},
legend style={font=\tiny},
declare function={ 
rhs(\t,\MM,\GG,\AA) = -sqrt( 2*\MM/\t + \GG)*tan(\AA r);
lhs(\t,\MM,\GG) = tan( sqrt( \t*(2*\MM + \t*\GG) ) r );
}
]
\pgfplotsset{
    % this modifies the `every tick label' style
    tick label style={ font=\footnotesize},
}

\addplot[blue!80!red, domain=0:8, samples=300]{rhs(x, 4,1,0)};
% \addlegendentry{\footnotesize $\alpha=0$}
\addplot[blue!60!red, domain=0:8, samples=300]{rhs(x, 4,1, pi/4)};
% \addlegendentry{\footnotesize $\alpha=\frac{\pi}{4}$}
\addplot[blue!20!red, domain=0:8, samples=300]{rhs(x, 4,1,3/8*pi)};
% \addlegendentry{\footnotesize $\alpha=\frac{3\pi}{8}$}
\addplot[blue!10!red, domain=0:8, samples=300]{rhs(x, 4,1, 5/12*pi)};
% \addlegendentry{\footnotesize $\alpha=\frac{5\pi}{12}$}
\addplot[red!90!green, domain=0:8, samples=300]{rhs(x, 4,1, pi/2+1/16*pi)};
% \addlegendentry{\footnotesize $\alpha=\frac{9\pi}{16}$}
\addplot[red!80!green, domain=0:8, samples=300]{rhs(x, 4,1, pi/2+1/8*pi)};
% \addlegendentry{\footnotesize $\alpha=\frac{5\pi}{8}$}
\addplot[red!40!green, domain=0:8, samples=300]{rhs(x, 4,1, pi/2+3/8*pi)};
% \addlegendentry{\footnotesize $\alpha=\frac{7\pi}{8}$}

\addplot[black, domain=0:1, samples=300, restrict y to domain=-50:50] {lhs(x, 4,1)};
\addplot[black, domain=1:3, samples=300, restrict y to domain=-50:50] {lhs(x, 4,1)};
\addplot[black, domain=3:8, samples=300, restrict y to domain=-50:50] {lhs(x, 4,1)};
\end{axis}

\end{tikzpicture}

\caption{$\gamma=1$.}
\end{subfigure}

\begin{subfigure}{.48\linewidth}
%%%% Gamma = -1, 0 < t < M
\begin{tikzpicture}[scale=.8, transform shape]
\begin{axis}[ymin=-20, ymax=35, xmin=-.15, xmax=4.5,
axis x line=center,
axis y line=center,
xlabel={$t$},
xlabel style={right},
xtick={2,4},
ytick=\empty,
legend style={at={(1.4,0.6)},
   anchor=north west,legend columns=3},
legend style={font=\small},
declare function={ 
rhs(\t,\MM,\GG,\AA) = -sqrt( 2*\MM/\t + \GG)*tan(\AA r);
lhs(\t,\MM,\GG) = tan( sqrt( \t*(2*\MM + \t*\GG) ) r );
}
]
\pgfplotsset{
    % this modifies the `every tick label' style
    tick label style={ font=\footnotesize},
}

\addplot[blue!80!red, domain=0:4, samples=300]{rhs(x, 4,-1,0)};
\addlegendentry{\footnotesize $\alpha=0$}
\addplot[blue!60!red, domain=0:4, samples=300]{rhs(x, 4,-1, pi/4)};
\addlegendentry{\footnotesize $\alpha=\frac{\pi}{4}$}
\addplot[blue!20!red, domain=0:4, samples=300]{rhs(x, 4,-1,3/8*pi)};
\addlegendentry{\footnotesize $\alpha=\frac{3\pi}{8}$}
\addplot[blue!10!red, domain=0:4, samples=300]{rhs(x, 4,-1, 5/12*pi)};
\addlegendentry{\footnotesize $\alpha=\frac{5\pi}{12}$}
\addplot[red!90!green, domain=0:4, samples=300]{rhs(x, 4,-1, pi/2+1/16*pi)};
\addlegendentry{\footnotesize $\alpha=\frac{9\pi}{16}$}
\addplot[red!80!green, domain=0:4, samples=300]{rhs(x, 4,-1, pi/2+1/8*pi)};
\addlegendentry{\footnotesize $\alpha=\frac{5\pi}{8}$}
\addplot[red!40!green, domain=0:4, samples=300]{rhs(x, 4,-1, pi/2+3/8*pi)};
\addlegendentry{\footnotesize $\alpha=\frac{7\pi}{8}$}

\addplot[black, domain=0:4, samples=300, restrict y to domain=-50:50] {lhs(x, 4,-1)};
\end{axis}

\end{tikzpicture}

\caption{$\gamma=-1$.}
\end{subfigure}

\caption{Plot of the left hand side (black) and right hand side (colors) of \eqref{eq:neweigenvalue} for various values of $\alpha$ and $\gamma\in\{0,1,-1\}$. 
At any $t$ where a curve for fixed $\alpha$ intersects the black curve, a new eigenvalue enters the spectral gap from $M$.
}
\label{fig:neweigenvalue}
\end{figure}
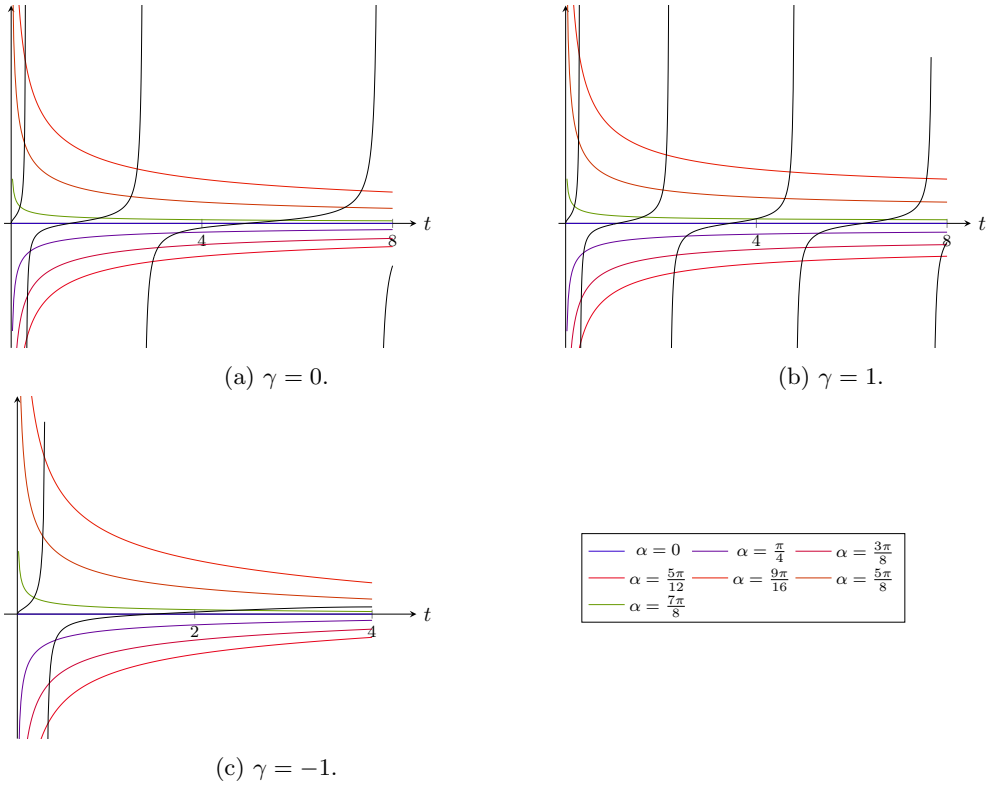

For $\gamma\in \{0, 1\}$, the condition $\nuint(\lambda,t,\gamma)>0$ implies that $T_\alpha(t)$ can only have eigenvalues in $(M-t,M)$. Moreover, Theorem~\ref{thm:MonotonOfEigenvalues} tells us that these eigenvalues are decreasing in $t$, so once they enter the gap from $M$, they decrease but they cannot touch $M-t$. If $\alpha = 0$, then a new eigenvalue enters the essential spectral gap from $M$ if $t(2M+\gamma t) \in \{ ( n\pi)^2: n\in\N_0\}$.
If $\alpha = \frac{\pi}{2}$, then a new eigenvalue enters the essential spectral gap from $M$ if $t(2M+\gamma t) \in \{ ( (n+\frac{1}{2})\pi )^2: n\in\N_0\}$.
These are the values from \eqref{eq:ToyExEValueCondition} for even respectively odd $n\in\N_0$.
When $\alpha$ increases from $0$ to $\pi$, the corresponding $t_n(\alpha)$ is mapped by the increasing function $t\mapsto t(2M+\gamma t)$ 
from $(n\pi)^2$ to $((n + 1)\pi)^2$
with $ t_n(\frac{\pi}{2}) \mapsto((n+ \frac{1}{2})\pi)^2$. This was to be expected since the eigenvalues are monotonically decreasing in the boundary condition parameter $\alpha$. Note also that for $\alpha \in (0, \frac{\pi}{2})$ the first eigenvalue does not enter from the essential spectrum, but is already present for $t=0$ at $\lambda(\alpha) = M\cos(2\alpha)$.
\smallskip

For $\gamma = -1$, a new eigenvalue enters the essential spectral gap from $-M$ when
\begin{equation*}
  \tan \left(\sqrt{ t(2M - t)}\right) = 
  \sqrt{\frac{ t }{ 2M - t }}\, \tan\alpha.
\end{equation*}
\end{example}

\clearpage

\bibliography{lit}
\bibliographystyle{alpha}

\end{document}